\documentclass[a4paper,12pt]{article}

\RequirePackage{amsthm,amsmath,amsfonts,amssymb}
\RequirePackage[numbers]{natbib}
\RequirePackage{graphicx}

\usepackage{color,ulem,tikz}

\theoremstyle{plain}
\newtheorem{theorem}{Theorem}[section]
\newtheorem{proposition}[theorem]{Proposition}
\newtheorem{lemma}[theorem]{Lemma}
\newtheorem{corollary}[theorem]{Corollary}

\theoremstyle{remark}
\newtheorem{example}{Example}
\newtheorem{definition}[theorem]{Definition}
\newtheorem{remark}[theorem]{Remark}

\begin{document}

\begin{center}

  \Large
  {\bf
    Asymptotic bias reduction of maximum likelihood
    estimates via penalized likelihoods with differential
    geometry}

  \normalsize

  \bigskip By \bigskip

  \textsc{Masayo Y. Hirose}
  \footnote{Institute of Mathematics for Industry,
  Kyushu University, Fukuoka 819-0395, Japan;
  E-mail: masayo@imi.kyushu-u.ac.jp}
  and
  \textsc{Shuhei Mano}
  \footnote{The Institute of Statistical Mathematics,
  Tokyo 190-8562, Japan; E-mail: smano@ism.ac.jp}


\end{center}

\bigskip

\small

{\bf Abstract.}
A method for asymptotic bias reduction of maximum likelihood
estimates of generic estimands is developed. The estimator
is realized as a plug-in estimator, where the parameter
maximizes the penalized likelihood with a penalty function
that satisfies a quasi-linear partial differential equation
of the first order. The integration of the partial differential
equation with the aid of differential geometry is discussed.
Applications to generalized linear models, linear
mixed-effects models, and a location-scale family are presented.

\smallskip

{\it Key Words.}
Bias reduction, information geometry, Jeffreys prior,
partial differential equation, plug-in estimator, shrinkage

\smallskip

2020 {\it Mathematics Subject Classification Numbers.}
Primary 62F12; Secondary 62B11, 62H12

\normalsize

\section{Introduction}
\label{sect:intr}

For a sample space $\mathcal{X}$, consider a parametric model
$M$, or a family of probability densities $p(\cdot;\xi)$ with
$d$-dimensional parameter $\xi=\{\xi^1,\ldots,\xi^d\}\in\Xi$,
where the parameter space $\Xi$ is an open subset of $\mathbb{R}^d$,
$d\ge 1$. The indices follow those of tensors otherwise stated;
the upper index should not be confused with power. 
We assume that $p(x;\xi)$ is a $C^\infty$-function, that is,
an infinitely-differentiable function of $\xi$ for each
$x\in\mathcal{X}$.

For a given real-valued function $f$ of $\xi$, where
$f:\Xi\to\mathbb{R}$, we call the parameterization, that is,
the result of the mapping $f(\xi)$, an {\it estimand},
and $f$ the {\it estimand function}.
The function $f$ may depend on quantities other
  than $\xi$, but such quantities are constants in estimating
  $f(\xi)$ and not shown explicitly.
In this paper, we discuss the estimation of an estimand
for a given estimand function.
An estimator $\delta(x)$, $x\in\mathcal{X}$ of an estimand
$f(\xi)$ is {\it unbiased} if
$\mathbb{E}_\xi\delta(x)=f(\xi),$ $\forall \xi\in\Xi$,
where the expectation $\mathbb{E}_\xi$ is with respect to
$p(\cdot;\xi)$. If an unbiased estimator of $f(\xi)$ exists,
the estimand $f(\xi)$ is said to be {\it $U$-estimable.}
The unbiased estimator $\delta$ of $f(\xi)$ is
the {\it uniform minimum variance unbiased estimator}
(UMVUE) of $f(\xi)$ if
$\text{var}_\xi\delta(x)\le\text{var}_\xi\delta'(x),$
$\forall \xi\in\Xi$, where $\delta'$ is any other unbiased
estimator of $f(\xi)$.

This paper presents a method for asymptotic bias reduction
of maximum likelihood estimates of generic estimands.
The resulting estimator asymptotically coincides with the UMVUE
if a complete sufficient statistic exists.

In a regular model with parameter $\xi$, the asymptotic bias
of the maximum likelihood estimate $\hat{\xi}_{\rm MLE}$ can
be expressed as
\begin{equation}\label{bias_MLE}
  b(\xi):=\mathbb{E}_\xi\hat{\xi}_{\rm MLE}-\xi
  =\frac{b_1(\xi)}{n}+\frac{b_2(\xi)}{n^2}+\cdots,
\end{equation}
where $n$ is usually interpreted as the number of observations,
or some other measure of the rate at which information accrues.
The maximum likelihood estimate is derived as the solution to
the system of score equations:
\[
  u_i(\xi;x):=\partial_i l(\xi;x)=0, \quad
  \partial_i:=\frac{\partial}{\partial\xi^i}, \quad
  i\in\{1,\ldots,d\},
\]
where
$l(\xi;x)=l(\xi;x_1,\ldots,x_n):=\sum_{i=1}^n\log p(x_i;\xi)$
is the log-likelihood of a sample
$(x_1,\ldots,x_n)\in\mathcal{X}^n$
and $u_i(\xi;x)$ are the log-likelihood and score functions,
respectively. Firth \cite{Fir93} proposed a method to remove
the $O(n^{-1})$ term from \eqref{bias_MLE}. His bias-reduced
estimator is the solution to the modified score equations:
\begin{equation}\label{firth_score}
  u^*_i(\xi;x):=u_i(\xi;x)-\sum_{j=1}^d
  \kappa_{i,j}(\xi)b_1^j(\xi)=0, \quad i\in\{1,\ldots,d\},
\end{equation}
where $\kappa_{i,j}(\xi)=n^{-1}\mathbb{E}_\xi[u_iu_j]$.
 
His proposal was inspired by a geometrical interpretation of
the modified score equations \eqref{firth_score} in
one-dimensional exponential families with canonical
parameterization. For the log-likelihood
$l(\xi;t)=t\xi-\psi(\xi)$ with the canonical parameter $\xi$
and sufficient statistic $t$, the score function is given by
$u(\xi;t)=t-\psi'(\xi)$. The bias $b(\xi)$ of
$\hat{\xi}_{\rm MLE}$ arises from the combination of
the unbiasedness of the score function
$\mathbb{E}_\xi u(\xi;t)=0$ and the curvature of the score
function, $u''(\xi;t)\neq 0$. Because
\[
  0=u(\hat{\xi}_{\rm MLE};t)
  =u(\xi;t)+u'(\xi;t)(\hat{\xi}_{\rm MLE}-\xi)
  +\frac{1}{2}u''(\xi;t)(\hat{\xi}_{\rm MLE}-\xi)^2+\cdots,
\]
if $u(\xi;t)$ is linear in $\xi$, then there is no bias.
Otherwise, $\hat{\xi}_{\rm MLE}$ subjects to bias.  Firth's
idea is to shift the score function downward at each point
$\xi$ by an amount $i(\xi)b(\xi)$, where
$-i(\xi)=u'(\xi;t)=-\psi''(\xi)$ is the gradient of the score
function at $\xi$ (see Figure 1 of \cite{Fir93}); this
defines a modified score function
$u^*(\xi;t):=u(\xi;t)-i(\xi)b(\xi)$. Hence, a modified
estimate $\hat{\xi}$ is given as a solution to $u^*(\xi;t)=0$.
Here, $-i(\xi)b(\xi)$ corresponds to the second term on
the right-hand side of \eqref{firth_score} in one-dimensional
exponential families.

Firth \cite{Fir93} argued that in multidimensional exponential
families with canonical parameterization, his method may be
regarded as a penalized maximum likelihood estimation with
penalized likelihood
$l^*(\xi;t):=l(\xi;t)+\tilde{l}(\xi)$, where the non-random
penalty function $\tilde{l}(\xi)$ satisfies the following
system of partial differential equations
\begin{equation}\label{firth_pen}
  \partial_i\tilde{l}(\xi)
  =-\sum_{j=1}^d\kappa_{i,j}(\xi)b_1^j(\xi) \quad
  \text{for~all} \quad i\in\{1,\ldots,d\}.
\end{equation}
He obtained $\tilde{l}(\xi)=\log\sqrt{{\rm det}i(\xi)}$ as
a solution of \eqref{firth_pen}, where $i(\xi)$ is
the information matrix. Firth pointed out that the penalty
function coincides with the logarithm of the Jeffreys prior.

In practice, however, there often are cases in which we are
interested in a function of the parameter rather than
the parameter itself. A standard logistic regression has
problems of stability when the event frequency is low,
and Firth's bias-reduction method has been a tool to stabilize
the logistic regression \citep{Kos14}. Firth's method is
for estimates of regression coefficients. However, it is
natural to ask what we should do when we are interested
in the event frequency. Unfortunately, removing the bias
in estimates of regression coefficients introduces bias in
the plug-in estimator of the event frequency. Elgmati et al.
\cite{Elg15}
tackled this problem and proposed an approach intermediate
between Firth and no penalty as a pragmatic compromise
between stability and bias.

Another example in small-area estimation is to estimate
the magnitude of shrinkage of a direct estimator of the mean
of each area towards the synthetic estimator like
a ground mean, where a direct estimator is constructed based
on a sample only within each area. A measure of the magnitude
is the proportion of the error variance to the marginal
variance of the direct estimator. Based on a penalized
likelihood approach, Hirose and Lahiri \cite{HL18}
successfully obtained a plug-in estimator of the measure
when the variance is known. The bias of their estimator is
asymptotically $o(n^{-1})$. Hirose and Lahiri \cite{HL21}
discussed a case of variance of the error variance is not
known, but provided no solution.

In this paper, we address the following two questions.
1) How to obtain an asymptotically unbiased estimator for
generic estimands?
2) In Firth's bias-reduction method for multidimensional
models, the system of partial differential equations
\eqref{firth_pen} appears. However, a system of partial
differential equations is not always integrable. Under what
situations is the system integrable? To the best of
the authors' knowledge, the system \eqref{firth_pen} appeared
in the literature many years ago, such as \cite{Har64}, but
the integrability has
rarely been cared for. An exception is the result obtained by
Kosmidis and Firth
\cite{KF09} for estimation of parameters in generalized
linear models. Note that the multidimensional setting is
essential; if a model is one-dimensional, we may apply
Firth's method to the model reparameterized by $f(\xi)$.
The system \eqref{firth_pen} reduces to an ordinary
differential equation, and the integrability issue disappears.

We consider a generic estimand function $f$ represented as
a scalar function of a multidimensional parameter $\xi\in\Xi$.
We employ the \textit{plug-in estimator}
$f(\hat{\xi}(x))=f\circ\hat{\xi}(x)$, where $\hat{\xi}(x)$,
$x\in\mathcal{X}$ maximizes the suitable penalized likelihood.
Kosmidis \cite{Kos14} classified methods of bias reduction as
explicit methods and implicit methods. The proposed method is
an implicit method: we evaluate the bias of the estimator
$f(\hat{\xi})$ at the solution $\hat{\xi}$ of an equation.
Explicit methods, such as the jackknife, rely on a one-step
procedure, where the bias of the estimator
$f(\hat{\xi}_{\rm MLE})$ is evaluated at $\hat{\xi}_{\rm MLE}$
and subtracted from $f(\hat{\xi}_{\rm MLE})$ to give the new
estimate. Our method has practical advantages to explicit
methods: the estimator may exist even if $\hat{\xi}_{\rm MLE}$
does not exist, and the range is in that of $f$ as long as
$\hat{\xi}\in\Xi$. Explicit methods may suffer from giving
an unrealistic estimate because it could be out of the range
of $f$ by the subtraction. See \cite{HL18} for a discussion
on positive $f$.

The remainder of this paper is organized as follows.
Section~\ref{sect:bias} presents our method for asymptotic
bias reduction of maximum likelihood estimates of generic
models and estimands. The condition for {\it second-order
asymptotic unbiasedness,} i.e.,
$\mathbb{E}_\xi[f(\hat{\xi})]-f(\xi)=o(n^{-1})$ for large $n$,
requires that the penalty function satisfies a quasi-linear
partial differential equation of the first order.
We discuss the integration of the partial differential
equation for generic models and estimands.
For estimates of parameters, we show that our method
reproduces Firth's bias-reduction method.
In Section~\ref{sect:flat}, we discuss cases of flat model
manifolds.
We reproduce Kosmidis and Firth's \cite{KF09} result on the integrability of
the system of partial differential equations that appear
when we apply Firth's method to generalized linear models.
Several applications of our bias-reduction method to
problems involving various estimands,
including the problems in logistic regression and small
area estimation discussed above,
are presented in Section~\ref{sect:exam}.
Section~\ref{sect:disc} provides a discussion.

\section{Geometry of asymptotic bias reduction}
\label{sect:bias}

This section presents the main results of this study.
After preparing geometric concepts in
Section~\ref{subs:bias1}, we give our method for
asymptotic bias reduction in Section~\ref{subs:bias2}.
The method reduces to integration of a quasi-linear
partial differential equation for a suitable penalty
function,
as explained in Section~\ref{subs:bias3}.
Section~\ref{subs:bias4} presents a result in terms of
the geodesic distances.

\subsection{Differential geometric preliminaries}
\label{subs:bias1}

We can simplify multivariate statistical calculations with
tensors in differential geometry \citep{McC87}.
In addition to tensors, some concepts in differential
geometry help to display the following results in concise forms.
In this subsection, we
introduce the minimum of them in order to state
our main results.
For general background information on differential geometry,
see \cite{KN63}. A concise summary of the terminology in affine
differential geometry is Chapter I of \cite{NS94}.
Differential geometry of statistical model manifolds, known
as information geometry, is discussed in a book by
Amari and Nagaoka \cite{AN00} and a concise survey by Lauritzen
\cite{Lau87}.

For a parametric model $M=\{p(\cdot;\xi):\xi\in\Xi\}$,
mapping $p\mapsto\{\xi^1,\ldots,\xi^d\}$ for each point (i.e.,
a probability density) $p\in M$ can be regarded as a local
coordinate of $M$. If we consider parameterizations that are
$C^\infty$-diffeomorphic to each other as equivalent, then $M$
may be considered as a $C^\infty$-differentiable manifold.
In this sense, a parameterization of $M$ is a local coordinate
system of $M$.

For each point $\xi\in \Xi$, we define a Riemannian
metric called the Fisher metric, which is a tensor with components
\begin{equation}
  g_{ij}(\xi):=-\mathbb{E}_\xi[\partial_i\partial_jl(\xi;x)],
  \quad
  \partial_i:=\frac{\partial}{\partial\xi^i}, \quad
  i,j\in\{1,\ldots,d\},
  \label{metric}
\end{equation}
in a local coordinate system, where
$l(\xi;x):=\sum_{i=1}^n\log p(x_i;\xi)$ is the log-likelihood of
a sample $(x_1,\ldots,x_n)\in\mathcal{X}^n$ and the expectation
$\mathbb{E}_\xi$ is with respect to the probability density
$p(\cdot;\xi)$.
The differential operator of a variable with an
  upper index has a lower index.
We assume the Fisher metric tensor is an invertible matrix
everywhere. The inverse matrix of $g_{ij}$ is denoted by $g^{ij}$,
where $g_{ij}$ converts an upper index into a lower index,
while $g^{ij}$ converts a lower index into an upper index.
The determinant of the matrix $g_{ij}$ is denoted by $g$.
The Fisher metric tensor is the null cumulant of derivatives of
the log-likelihood denoted by $\kappa_{i,j}$ in Chapter 7 of
\cite{McC87}.

A tangent vector (or simply, a vector) is represented as
$X=\sum_ix^i\partial_i$, where $x^i$ are components with respect to
the local coordinate system $\{\xi^1,\ldots,\xi^d\}$. The set
of tangent vectors at $p\in M$, denoted by $T_pM$, is called
the {\it tangent space} of $M$ at $p$. The metric tensor
defines the inner product of vector fields $X=\sum_ix^i\partial_i$
and $Y=\sum_iy^i\partial_i$, as in
$\langle X,Y\rangle=\sum_{i,j}\langle\partial_i,\partial_j\rangle x^iy^j
:=\sum_{i,j}g_{ij}x^iy^j$.

An \textit{affine connection} on a manifold $M$ is a rule
of covariant differentiation on $M$, denoted by $\nabla$
(see Chapter III of \cite{KN63} or Section I.3 of \cite{NS94}
for affine connections). We write
\begin{equation*}\label{coder}
  \nabla_i\partial_j=\sum_k\Gamma^k_{ij}\partial_k,
\end{equation*}
where the system of functions $\Gamma^k_{ij}$ are
the \textit{Christoffel symbols} for the affine connection relative
to the local coordinate system (see Proposition III.7.4 of
\cite{KN63}), and $\nabla$ also refers to an affine connection.
In this paper, we specifically consider a one-parameter family of
affine connections called the $\alpha$-\textit{connection} \citep{AN00}
denoted by $\Gamma^{(\alpha)}_{ij,k}$, where
\begin{equation}\label{al_con}
  \Gamma^{(\alpha)}_{ij,k}
  :=\mathbb{E}_\xi[(\partial_i\partial_j l)u_k]+
  \frac{1-\alpha}{2}S_{ijk}, \quad \alpha\in\mathbb{R}
\end{equation}
for $i,j,k\in\{1,\ldots,d\}$. Here, $u_k(\xi;x):=\partial_kl(\xi;x)$
are the score functions, and
${\Gamma^{(\alpha)}}_{ij,k}=\sum_l g_{kr}{\Gamma^{(\alpha)}}^r_{ij}$
are the symbols constructed by converting the upper index of
the Christoffel symbols to a lower index.
The covariant differentiation with respect to
the $\alpha$-connection is specifically denoted by
$\nabla^{(\alpha)}$. Here, the symmetric tensor of order
three, $S_{ijk}:=\mathbb{E}_\xi[u_iu_ju_k]$ is called
the {\it skewness tensor}. The skewness tensor is a null
cumulant denoted by $\kappa_{i,j,k}$ in \cite{McC87}.
The vector obtained by the contraction
$S_i=\sum_{j,k}g^{jk}S_{ijk}$ appears. 

We discuss Riemannian manifolds identified with a parametric
models equipped with the Fisher metric tensors \eqref{metric}
and $\alpha$-connections \eqref{al_con}. We call such manifolds
{\it model manifolds.}

A useful relationship among $\alpha$-connections is
\begin{equation}\label{al_con2}
  \Gamma^{(\alpha_1)}_{ij,k}-\Gamma^{(\alpha_2)}_{ij,k}
  =\frac{\alpha_2-\alpha_1}{2}S_{ijk}, \quad
  {\rm for~all} \quad
  \alpha_1,\alpha_2\in\mathbb{R}.
\end{equation}
The covariant derivative of the metric tensor is as follows.
The proof is in Appendix~A.1.

\begin{proposition}
  The covariant derivative of the metric tensor with
  $\alpha$-connection is
  \begin{equation}\label{cov_met}
    \nabla^{(\alpha)}_i g_{jk}
    =\partial_i g_{jk}-\Gamma^{(\alpha)}_{ij,k}
    -\Gamma^{(\alpha)}_{ik,j}=\alpha S_{ijk}.
  \end{equation}  
\end{proposition}
\color{black}
The Christoffel symbols of the 0-connection (also called
the Levi-Civita connection or the Riemannian connection)
are represented by the derivatives of the metric tensors.
Since the middle expression of \eqref{cov_met} should be
zero if $\alpha=0$, we have
\begin{equation}\label{0_con}
  \Gamma^{(0)}_{ij,k}=\frac{1}{2}
  (\partial_i g_{jk}+\partial_j g_{ki}-\partial_k g_{ij}).
\end{equation}

In this paper, we define the operator $\Delta^{(\alpha)}$ acting
on a scalar field given by a function $f(\xi)$, $\xi\in\Xi$:
\begin{equation}\label{al_Lap}
  \Delta^{(\alpha)}f:=\sum_i\nabla^{(\alpha)i}\nabla^{(\alpha)}_if
  =\sum_{i,j}g^{ij}\partial_i\partial_j f
   -\sum_{i,j,k,r}g^{ij}g^{kr}\Gamma_{kr,i}^{(\alpha)}\partial_jf
\end{equation}
and call it the $\alpha$-\textit{Laplacian}. The 0-Laplacian
reduces to the conventional Laplacian, called
the Laplace--Beltrami operator, that satisfies
\begin{equation}\label{0_Lap}
  \Delta^{(0)}f=\frac{1}{\sqrt{g}}\sum_i\partial_i
  (\sqrt{g}\sum_jg^{ij}\partial_j f).
\end{equation}
If a function $f$ satisfies $\Delta^{(\alpha)}f=0$ for some
$\alpha$, we say that $f$ is {\it $\alpha$-harmonic}.
The authors did not find the notion of $\alpha$-harmonicity in
the literature, however, the $\alpha$-Laplacian is an example of
an extension of the Laplace--Beltrami operator proposed by
Eguchi and Yanagimoto \cite{EY08}. From
the relationship among $\alpha$-connections \eqref{al_con2}, we have
a useful relationship among $\alpha$-Laplacians:
\begin{equation}\label{al_Lap_rel}
  (\Delta^{(\alpha_1)}-\Delta^{(\alpha_2)})f=
  \sum_{i,j,k,r}g^{ij}g^{kr}
  (\Gamma^{(\alpha_2)}_{kr,i}-\Gamma^{(\alpha_1)}_{kr,i})
  \partial_jf
  =\frac{\alpha_1-\alpha_2}{2}\sum_jS^j\partial_jf
\end{equation}
for all $\alpha_1,\alpha_2 \in \mathbb{R}$.

\subsection{The condition for asymptotic unbiasedness}
\label{subs:bias2}

The first assertion of the following lemma comes from the bias
of the maximum likelihood estimator of a parameter.
The result seems classic and can be found in the literature
(e.g., Equation 20 of \cite{CS68} and Section 7.3 of \cite{McC87}),
but the authors did not find a rigorous form. Therefore,
we present the assertion with regularity conditions.
The proof is in Appendix~A.2.

\begin{lemma}\label{lemm:exps}
  Under the regularity conditions B1-B6 in Appendix~A.2,
  for a penalty function $\tilde{l}(\xi)=O(1)\in C^4$ of
  parameters $\xi$ for large $n$, where $C^4$ is
  the set of four-times differentiable functions, we have
\begin{itemize}
  
\item[$(i)$]
  $\displaystyle\mathbb{E}_\xi[(\hat{\xi}-\xi)^i]=\sum_jg^{ij}(\partial_j \tilde{l}{-\frac{1}{2}\sum_{k,r}g^{kr}\Gamma^{(-1)}_{kr,j}})+o(n^{-1})$, and
\item[$(ii)$]
  $\displaystyle\mathbb{E}_\xi[(\hat{\xi}-\xi)^i(\hat{\xi}-\xi)^j]=g^{ij}+o(n^{-1})$,

\end{itemize}
for large $n$, where $\hat{\xi}$ maximizes the penalized likelihood
$l^*(\xi;x)=l(\xi;x)+\tilde{l}(\xi)$ in $\xi$, where
  $l(\xi;x)$ is the log-likelihood.

\end{lemma}

\begin{remark}\label{rema:fisher}
  Adding the penalty function denormalizes the probability
  density, resulting in the density $p(\cdot;\xi)e^{\tilde{l}(\xi)/n}$.
  The maximization process is on the denormalized model
  (e.g., Fisher's scoring in \eqref{scoring}), but the expectations
  are with respect to the original model. Therefore, the Fisher metric
  tensors and $\alpha$-connections appear below are of the original
  model manifolds.
\end{remark}

Lemma~\ref{lemm:exps} (i) reveals that the $O(n^{-1})$ term
of the bias of the $i$-th component of the maximum likelihood
estimator of the parameter $\hat{\xi}_{\rm MLE}$ is
$-\sum_jg^{ij}(\sum_{k,r}g^{kr}\Gamma^{(-1)}_{kr,j})/2$.
\color{black}
Firth's \cite{Fir93} condition \eqref{firth_pen} is expressed as
\begin{equation}\label{firth_pen_geo}
  \partial_i\tilde{l}-\frac{1}{2}\sum_{k,r}g^{kr}\Gamma^{(-1)}_{kr,i}=0
  \quad \text{for~all} \quad i\in\{1,\ldots,d\}.
\end{equation}
The penalty function $\tilde{l}$ satisfying \eqref{firth_pen_geo}
removes the biases of all components of $\hat{\xi}_{\rm MLE}$
simultaneously.

Let us discuss bias reduction of generic estimands.
Suppose an estimand function $f:\Xi\to\mathbb{R}$ is given.
Under the regularity conditions, the bias of the plug-in
estimator $f(\hat{\xi})$ of an estimand $f(\xi)$ is given as 
\begin{align*}
  \mathbb{E}_\xi[f(\hat{\xi})-f(\xi)]=&
  \sum_i\mathbb{E}_\xi[(\hat{\xi}-\xi)^i]\partial_i f+
  \sum_{i,j}\mathbb{E}_\xi[(\hat{\xi}-\xi)^i(\hat{\xi}-\xi)^j]
  \frac{1}{2}\partial_i\partial_jf\\
  &+\sum_{i,j,k}\mathbb{E}_\xi[(\hat{\xi}-\xi)^i(\hat{\xi}-\xi)^j(\hat{\xi}-\xi)^k]
  \frac{1}{3!}\partial_i\partial_j\partial_kf+o(n^{-4/3}).
\end{align*}
Lemma~\ref{lemm:exps} (i) implies that if we
expand the penalty function $\tilde{l}(\xi)=O(1)$ for large
$n$ in fractional powers of $n$ and choose it to cancel out
the bias of $f(\hat{\xi})$ for each order in $n$, then we can
remove the bias up to the desired order. We demonstrate this
scheme up to $O(n^{-1})$ in the following lemma. The proof is
in Appendix~A.1.
The geometric meaning of the gradient of function $f$ denoted by
${\rm grad} f=\sum_{i,j}g^{ij}\partial_j f\partial_i$ will
appear in Section~\ref{subs:bias3}.

\begin{lemma}\label{lemm:cond}
  Consider a parametric model $M=\{p(\cdot;\xi):\xi\in\Xi\}$
  and an estimand function $f\in C^3$. Under the regularity conditions
  of Lemma~\ref{lemm:exps}, for a $U$-estimable estimand $f(\xi)$,
  $\xi\in\Xi$, the bias of the plug-in estimator $f(\hat{\xi})$
  of $f(\xi)$, where $\hat{\xi}$ maximizes the penalized
  log-likelihood, is $o(n^{-1})$ if and only if $\tilde{l}$ satisfies
  \begin{equation}\label{cond}
    \langle{\rm grad}f,{\rm grad}\tilde{l}\rangle
    +\frac{1}{2}\Delta^{(-1)}f=o(n^{-1}).
  \end{equation}  
\end{lemma}

\begin{remark}\label{rema:mse}
  An estimator with an asymptotic bias of $o(n^{-1})$ is said
  to be second-order unbiased. Lemma~\ref{lemm:exps} (ii)
  implies that any choice of the penalty function $\tilde{l}$ makes
  the plug-in estimator second-order efficient
  (saturates the Cram\'{e}r--Rao bound with error $o(n^{-1})$).
  Lemma~\ref{lemm:cond} specifies second-order unbiased estimator
  within the class of the second-order efficient estimators.
  A stronger result is Theorem~\ref{theo:UMVUE2}.
\end{remark}

If we have a complete sufficient statistic, an optimality
result of the plug-in estimator in terms of the asymptotic
bias follows. The main results of this paper are stated in
the following theorem. The proof is in
Appendix~A.1.

\begin{theorem}\label{theo:UMVUE2}
  Let a sample be distributed according to a parametric model
  $M=\{p(\cdot;\xi):\xi\in\Xi\}$, and suppose there exists
  a complete sufficient statistic for $M$. Then, under
  the assumptions of Lemma~\ref{lemm:cond}, the plug-in
  estimator $f(\hat{\xi})$ of an estimand $f(\xi)$,
  $\xi\in\Xi$, where $\hat{\xi}$ maximizes the penalized
  likelihood with $\tilde{l}$ satisfying
  the condition \eqref{cond}, coincides with the UMVUE of
  $f(\xi)$ up to $O(n^{-1})$.
\end{theorem}

For a given estimand function $f$, the condition \eqref{cond}
is fulfilled if we can find a penalty function $\tilde{l}$
that satisfies the following first-order quasi-linear partial
differential equation for $\tilde{l}$:
\begin{equation}\label{pde}
  \langle {\rm grad}f,{\rm grad}\tilde{l}\rangle
  +\frac{1}{2}\Delta^{(-1)}f=0.
\end{equation}
Hence, our task now is integrating \eqref{pde} for $\tilde{l}$.
Here, ``quasi-linear'' means that \eqref{pde} is linear in
the partial derivatives $\partial_i\tilde{l}$ with terms that
do not depend on $\tilde{l}$. If $f$ is $(-1)$-harmonic,
$\Delta^{(-1)}f=0$, and \eqref{pde} is linear. In this case,
we achieve asymptotic unbiasedness without a penalty.

\begin{corollary}\label{coro:no_adjm}
  If an estimand function $f$ is $(-1)$-harmonic, then
  the bias of the plug-in estimator $f(\hat{\xi}_{\rm MLE})$
  of an estimand $f(\xi)$, $\xi\in\Xi$ is $o(n^{-1})$
  and coincides with the UMVUE up to $O(n^{-1})$ given that
  a complete sufficient statistic for the parametric model
  $M$ exists.
\end{corollary}


\begin{example}
  For an exponential family with canonical parameterization,
  the log-likelihood is $l(\xi;t)=\sum_i t_i\xi^i-\psi(\xi)$,
  where $t$ is the sufficient statistic. We have
  $\mathbb{E}_\xi\eta_i(\hat{\xi}_{\rm MLE})=\eta_i(\xi)$,
  $\forall\xi\in\Xi$ for the expectation parameter
  $\eta_i(\xi):=\mathbb{E}_\xi t_i=\partial_i\psi$
  (see Example 6.6.3 of \cite{LC98}). We have
  $g_{ij}=\partial_i\partial_j\psi=\partial_i\eta_j$ and
  $\Gamma^{(-1)}_{ij,k}=2\Gamma^{(0)}_{ij,k}=\partial_ig_{jk}+\partial_jg_{ik}-\partial_kg_{ij}=\partial_i\partial_j\partial_k\psi=\partial_i g_{jk}$,
  because $\mathbb{E}_\xi[(\partial_i\partial_jl)u_k]=-\partial_i\partial_j\psi\mathbb{E}_\xi[u_k]=0$ and the property \eqref{0_con}.
  Then, $\eta_{i_0}$ for an index $i_0\in\{1,\ldots,d\}$ is
  a $(-1)$-harmonic function, because
  \[
  \Delta^{(-1)}\eta_{i_0}
    =\sum_{i,j}g^{ij}\partial_i\partial_j\eta_{i_0}
    -\sum_{i,j,k,r}g^{ij}g^{kr}\partial_kg_{ri}\partial_j\eta_{i_0}
    =\sum_{i,j}g^{ij}\partial_i g_{ji_0}
    -\sum_{k,r}g^{kr}\partial_k g_{ri_0}=0
  \]
  follows by
  $\sum_{i,j}g^{ij}\partial_k g_{ri}\partial_j\eta_{i_0}=\sum_{i,j}g^{ij}\partial_kg_{ri}g_{ji_0}=\partial_k g_{ri_0}$.
  Therefore, Corollary~\ref{coro:no_adjm} concludes that
  the estimator $\eta_{i_0}(\hat{\xi}_{\rm MLE})$ of the estimand
  $\eta_{i_0}(\xi)$ coincides with the UMVUE up to $O(n^{-1})$.
  In fact, since $\eta_{i_0}(\hat{\xi}_{\rm MLE})$ is an unbiased
  estimator of $\eta_{i_0}(\xi)$ and is the complete sufficient
  statistic of the exponential family, the Lehmann-Scheff\'e theorem
  implies that $\eta_{i_0}(\hat{\xi}_{\rm MLE})$ is the UMVUE of
  $\eta_{i_0}(\xi)$ exactly for any $n\ge 1$.
\end{example}  

Before closing this subsection, we revisit the bias reduction of
$\hat{\xi}_{\rm MLE}$ discussed by Firth \cite{Fir93}.
When we remove the biases of all the components of
$\hat{\xi}_{\rm MLE}$ with a single penalty function $\tilde{l}$,
substituting \eqref{al_Lap} with $\alpha=-1$ into \eqref{pde}
yields
\begin{equation}\label{pde_firth}
  \sum_jg^{ij}(\partial_j\tilde{l}
  -\frac{1}{2}\sum_{k,r}g^{kr}\Gamma^{(-1)}_{kr,j})=0,
  \quad i\in\{1,\ldots,d\}
\end{equation}
for $f=\xi^1,\ldots,f=\xi^d$, respectively, and we obtain
\eqref{firth_pen_geo}, because  
\begin{align*}
0&=\sum_i g_{mi}\sum_jg^{ij}(\partial_j\tilde{l}
  -\frac{1}{2}\sum_{k,r}g^{kr}\Gamma^{(-1)}_{kr,j})=
  \sum_j\delta^j_m(\partial_j\tilde{l}
  -\frac{1}{2}\sum_{k,r}g^{kr}\Gamma^{(-1)}_{kr,j})\\
 &=\partial_m\tilde{l}-\frac{1}{2}\sum_{k,r}g^{kr}\Gamma^{(-1)}_{kr,m}
\end{align*}
for $m\in\{1,\ldots,d\}$.
Since \eqref{pde_firth} is the system of $d$ partial differential
equations for the single unknown function $\tilde{l}$, we have to
consider the integration of an overdetermined system of partial
differential equations, except for one-dimensional models $(d=1)$. 
An overdetermined system of partial differential equations is
not always integrable. We will discuss this issue in
Section~\ref{subs:flat2}.

\begin{remark}\label{rema:integ}
  We do not have to remove the biases of all the components
  of $\hat{\xi}_{\rm MLE}$ with a single penalty function.
  We may consider the bias reduction for each component of
  $\hat{\xi}_{\rm MLE}$ separately. In fact, if we consider
  the estimand function $f=\xi^{i_0}$ for an index
  $i_0\in\{1,\ldots,d\}$, the condition \eqref{pde} for
  the penalty function, say $\tilde{l}^{(i_0)}$, becomes
  $\sum_jg^{i_0j}(\partial_j\tilde{l}^{(i_0)}-\sum_{k,r}g^{kr}\Gamma^{(-1)}_{kr,j}/2)=0$.
  The integrability issue disappears, because this is a single
  partial differential equation for $\tilde{l}^{(i_0)}$.
  See the bias reduction of the sample variance
    in Section~\ref{subs:exam3} for an example.
\end{remark}

\subsection{Integration of the quasi-linear partial differential equation}
\label{subs:bias3}

In this subsection, we discuss the integration of the partial
differential equation \eqref{pde} for generic models and estimands.
A generic quasi-linear partial differential equation of
the first order is known to be equivalent to a system of ordinary
differential equations, and geometric interpretation makes this
equivalence apparent (see Sections I.4, I.5, II.2 of \cite{CH62}
and Chapter 17 of \cite{Lee02}).

Solving a partial differential equation is finding an integral
manifold.
An $r$-dimensional {\it distribution} $\mathcal{D}$
on a manifold $M$ is an assignment to each point $p\in M$
a subspace $\mathcal{D}_p$ of the tangent space $T_pM$.
The distribution here is a collection of subspaces, and should
not be confused with a probability distribution.
A connected submanifold $N$ of $M$ is called an {\it integral manifold} of
$\mathcal{D}$ if $T_pN=\mathcal{D}_p$ for all $p\in N$.
If an integral manifold exists through each point of $M$,
$\mathcal{D}$ is said to be {\it completely integrable.}
The maximal integral manifolds do not intersect each other and
completely cover the whole of $M$. We say that the maximal
integral manifolds of $\mathcal{D}$ form the {\it leaves} of
a {\it foliation} of $M$ (see Chapter 2 of \cite{Mor01} or
Chapter 19 of \cite{Lee02}).

For each function $f\in C^1$, the total derivative of $f$
at $p\in M$ is denoted as $Xf=\sum_ix^i\partial_if$ for $X\in T_pM$.
There exists a unique vector field denoted by ${\rm grad}f$,
called the {\it gradient} of $f$, satisfying
$\langle {\rm grad}f,X\rangle=Xf$. In a local coordinate system,
${\rm grad} f=\sum_{i,j}g^{ij}\partial_j f\partial_i$. The gradient
of $f$ is a normal vector field to each $(d-1)$-dimensional
hypersurface in the $d$-dimensional manifold $M$ on which $f$
is a constant. Such a hypersurface is called a {\it level surface}
of $f$ (see Section 4.1 of \cite{Mor01}). 

We assume that
${\rm grad}f$ in a model manifold $M$ is not degenerated everywhere:
$|{\rm grad}f|^2:=\langle{\rm grad}f,{\rm grad}f\rangle>0$.
Consider an embedding of $M$ in $M\times\mathbb{R}$ by
introducing a local coordinate system
$(\xi^1,\ldots,\xi^d,\xi^{d+1})$, where
$\xi^{d+1}=\tilde{l}(\xi^1,\ldots,\xi^d)$.
Suppose that a solution of the partial differential equation
\eqref{pde} for $\tilde{l}$ is given as an implicitization
of $\xi^{d+1}=\tilde{l}(\xi^1,\ldots,\xi^d)$ by the equation
\begin{equation}\label{level_phi}
  \phi(\xi^1,\ldots,\xi^{d+1})=\phi_0 
\end{equation}
for a constant $\phi_0$. This equation determines
a $d$-dimensional level surface of $\phi$ in
$M\times\mathbb{R}$, on which $\phi$ is the constant $\phi_0$,
and is the integral manifold determined by the partial
differential equation \eqref{pde}. In fact, the constancy gives
$\partial_i\phi+\partial_{d+1}\phi\cdot\partial_i\tilde{l}=0$
for all $i\in\{1,\ldots,d\}$, and multiplying each equation by
$({\rm grad}f)^i$ and summing them yields
\begin{align}
  &({\rm grad}f)^i\partial_i\phi+\partial_{d+1}\phi\sum_i({\rm grad}f)^i
  \partial_i\tilde{l}
  =\langle{\rm grad}f,{\rm grad}\phi\rangle+\partial_{d+1}\phi
  \langle{\rm grad}f,{\rm grad}\tilde{l}\rangle\nonumber\\
  &=\langle{\rm grad}f,{\rm grad}\phi\rangle-
  \partial_{d+1}\phi\Delta^{(-1)}f/2=
  \langle({\rm grad}f,-\Delta^{(-1)}f/2),
  {\rm grad}^*\phi\rangle=0, \label{monge}
\end{align}
where ${\rm grad}^*\phi$ is the gradient of $\phi$ 
defined in $M\times\mathbb{R}$ with components
$({\rm grad}^*\phi)^i=\sum_jg^{ij}\partial_j\phi$ for
$i\in\{1,\ldots,d\}$ and
$({\rm grad}^*\phi)^{d+1}=\partial_{d+1}\phi$.
The second-last equality follows by \eqref{pde}. Here,
${\rm grad}^*\phi$ is the normal vector field to each
integral manifold determined by \eqref{level_phi}.
The last equality of \eqref{monge} requires that at each
point tangent planes of all hypersurfaces through
the point belong to a single pencil of planes whose axis is
given by the direction field $({\rm grad}f,-\Delta^{(-1)}f/2)$
in $M\times\mathbb{R}$, called the {\it Monge axis} of
the partial differential equation \eqref{pde}.

The integral curves along this direction field are defined
by a system of ordinary differential equations and are called
the {\it characteristic curves} of the partial differential
equation \eqref{pde} (see Figure~\ref{fig:fig1}).
\begin{figure}
  \centering
  \includegraphics[height=50mm]{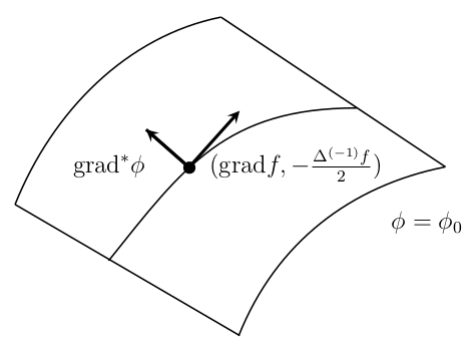}
    \caption{An integral manifold $\phi=\phi_0$, which encodes
    the penalty function $\tilde{l}$ by implicitization,
    and a characteristic curve on it. The normal vector
    field to the integral manifold, ${\rm grad}^*\phi$, and
    the Monge axis $({\rm grad}f,-\Delta^{(-1)}f/2)$ at
    a point on the characteristic curve are orthogonal.}
    \label{fig:fig1}
\end{figure}
In other words, a $d$-parameter family of characteristic curves
parameterized by $s$:
\[
  \{\xi^1(s),\ldots,\xi^d(s),
  \xi^{d+1}(s)=\tilde{l}(\xi^1(s),\ldots,\xi^{d}(s))\},
  \quad s\ge 0
\]
satisfies the system of ordinary differential equations:
\begin{equation}\label{odes}
  \frac{d\xi^i(s)}{ds}=({\rm grad}f)^i,
  \quad i\in\{1,\ldots,d\} \quad \text{and} \quad 
  \frac{d\xi^{d+1}(s)}{ds}=\frac{d\tilde{l}(s)}{ds}=
  -\frac{\Delta^{(-1)}f}{2},
\end{equation}
where the initial values $\{\xi^1(0),\ldots,\xi^{d}(0)\}$ and
$s$ are the parameters of the characteristic curves.
These parameters should not be confused with statistical
parameters. By virtue of the theory of ordinary differential
equations, the unique solution exists if the right-hand sides
of \eqref{odes} are Lipschitz continuous. Eliminating
the parameters from the expression yields an explicit
expression for the penalty function
$\tilde{l}=\tilde{l}(\xi^1,\ldots,\xi^d)$. This procedure is
discussed in Sections 2.1 and 2.2 of \cite{CH62}.
As in \eqref{monge}, we see that
\begin{equation}\label{pde_int}
  \langle{\rm grad}f,{\rm grad}\phi\rangle
  -\partial_{d+1}\phi\frac{\Delta^{(-1)}f}{2}=0.
\end{equation}
Thus, a solution of the system of ordinary differential equations
\eqref{odes} gives
\begin{equation*}
  \frac{d\phi}{ds}=\sum_i\partial_i\phi
  \frac{d\xi^i}{ds}+
  \partial_{d+1}\phi\frac{d\xi^{d+1}}{ds}=\langle{\rm grad}f,{\rm grad}\phi\rangle
  -\partial_{d+1}\frac{\Delta^{(-1)}f}{2}=0,
\end{equation*}
which means that along each characteristic curve
$\phi(\xi^1(s),\ldots,\xi^{d+1}(s))$, $s\ge 0$
is a constant value, as expected from \eqref{level_phi}.
In particular, if $f$ is $(-1)$-harmonic, the Monge axis
is always parallel to the model manifold $M$ and
$\xi^{d+1}=\tilde{l}(\xi^1,\ldots,\xi^d)$ has a constant value.
This is a geometric interpretation of why we do not need a penalty,
as seen in Corollary~\ref{coro:no_adjm}.



In principle, we can solve the partial differential equation
\eqref{pde} for generic model manifolds and estimand functions.
To obtain an explicit expression, we have to eliminate parameters
$s$ and $\{\xi^1(0),\ldots,\xi^d(0)\}$ from the solution of
the system of the ordinary differential equations
\eqref{odes}. However, this process is tedious in practice.
In fact, elimination of parameters to obtain an implicit
representation of a manifold is a topic in
computational algebraic geometry and is called
the {\it implicitization problem} (see Chapter 3 of \cite{CLO07}).
Further investigation of the implicitization
problem seems interesting, but in this study we developed
the following trick to circumvent the implicitization problem.

Suppose that we seek a solution to the partial differential
equation \eqref{pde_int} for an integral of the system of
ordinary differential equations \eqref{odes} of the form
\[
  \phi(\xi^1,\ldots,\xi^{d+1})
  ={\chi}(f(\xi^1,\ldots,\xi^d))-\xi^{d+1}
  ={\chi}(f(\xi^1,\ldots,\xi^d))-\tilde{l}(\xi^1,\ldots,\xi^d)
\]
with an injective map ${\chi}:\mathbb{R}\to\mathbb{R}$.
This form of $\phi$ has a geometric interpretation. Note that
on an integral manifold, that is, $\phi={\rm const.}$,
$f={\rm const.}$ if and only if $\tilde{l}={\rm const.}$
We have $(d-1)$-dimensional level surfaces of the estimand
function $f$ in the $d$-dimensional model manifold $M$,
namely, $f$ gives a codimension-1 foliation of $M$, and
the level surfaces of $f$ are the leaves. Let the foliation
be denoted by $\{N_w(f):w\in\mathbb{R}\}$ with
\[
  N_w(f)=\{(\xi^1,\ldots,\xi^d)\in M:f(\xi^1,\ldots,\xi^d)=w\}.
\]
On the other hand, on the $d$-dimensional integral surface
$\phi=u$ for a constant $u\in\mathbb{R}$, we have
$(d-1)$-dimensional level surfaces of $\tilde{l}$. Let
the foliation be denoted by
$\{N^*_{u,v}(\tilde{l}):u,v\in\mathbb{R}\}$ with
\[
  N^*_{u,v}(\tilde{l})=\{(\xi^1,\ldots,\xi^d,v)\in M\times\mathbb{R}:\phi(\xi^1,\ldots,\xi^d,v)=u,\tilde{l}(\xi^1,\ldots,\xi^d)=v)\}.
\]
We introduce a projection:
\[
  \pi(N^*_{u,v}(\tilde{l})):=\{(\xi^1,\ldots,\xi^d):(\xi^1,\ldots,\xi^{d},\xi^{d+1})\in N^*_{u,v}(\tilde{l})\}.
\]
Then, if we consider the integral manifold with $u=\chi(w)-v$,
we have
\[
  \{N_w(f):w\in\mathbb{R}\}
  =\{\pi(N^*_{\phi_0,\chi(w)-\phi_0}(\tilde{l})):w\in\mathbb{R}\}
  =\{\pi(N^*_{\phi_0,v}(\tilde{l})):v\in\mathbb{R}\},
\]
which means that the foliation of the integral manifold
$\phi=\phi_0$ by the level surfaces of $\tilde{l}$ projected
onto $M$ constitutes the foliation of $M$ by the level surfaces
of $f$ (Figure~\ref{fig:fig2}).
\begin{figure}
  \centering
  \includegraphics[height=55mm]{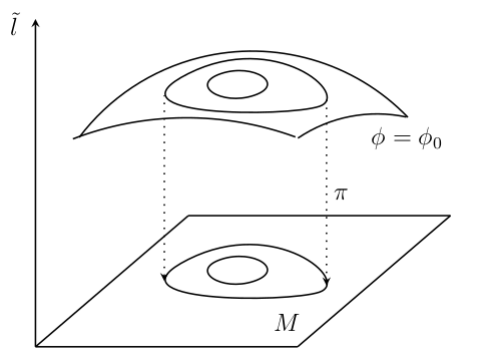}
  \caption{The foliation of the integral manifold $\phi=\phi_0$
    by the level surfaces of $\tilde{l}$ projected onto
    the model manifold $M$ constitute the foliation of $M$ by
    the level surfaces of $f$. The contours are the level surfaces.}
  \label{fig:fig2}
\end{figure}

The condition under which $\phi$ is an integral of the system
of ordinary differential equations \eqref{odes} can be
expressed as
\[
  \frac{d\phi}{ds}=\frac{d\chi}{d f}
  \sum_i\partial_i f\frac{d\xi^i}{ds}-\frac{d\xi^{d+1}}{ds}
  =\frac{d\chi}{df}|{\rm grad} f|^2+\frac{\Delta^{(-1)} f}{2}
  =0.
\]
Therefore, by integrating the differential equation
\begin{equation}\label{psi_cond}
  \frac{d\chi}{df}
  =-\frac{1}{2}\frac{\Delta^{(-1)} f}{|{\rm grad} f|^2},
\end{equation}
we obtain the penalty function
$\tilde{l}(\xi^1,\ldots,\xi^d)=\chi(f(\xi^1,\ldots,\xi^d))+{\rm const.}$
In particular, if the right-hand side of \eqref{psi_cond} is
a function of $f$, the integration is straightforward.
The constant can be chosen arbitrarily since the maximizer
$\hat{\xi}$ of the penalized likelihood is our concern and
does not depend on the constant. Therefore, the penalty
function is not unique.

\subsection{Functions of geodesic distances}
\label{subs:bias4}

We can obtain explicit results if an estimand is a function of
the squared geodesic distance. Our result is the following theorem.
The proof is in Appendix~A.3. The lemma used appeared in Riesz's
\cite{Rie49} work  on constructing the fundamental solution of
a partial differential equation as a convergent series in
the squared geodesic distance associated with the Laplace-Beltrami
operator \eqref{0_Lap}.
The definitions of the geodesic distance and the normal coordinate
system are in Appendix~A.3.

\begin{theorem}\label{theo:cond0}
  Consider a parametric model $M=\{p(\cdot;\xi):\xi\in\Xi\}$.
  Fix a point $\zeta\in M$ and consider the geodesic
  joining $\zeta$ and point $\xi\in U_\zeta$
  with the the geodesic distance $t(\xi):={\rm dis}(\zeta,\xi)$.
  Suppose we have an estimand function of the squared geodesic
  distance $f(\xi)=\varphi(t^2(\xi^1,\ldots,\xi^d))$.
  Then, under the regularity conditions of Lemma~\ref{lemm:cond},
  the bias of the plug-in estimator $f(\hat{\xi})$ of
  the estimand $f(\xi)$, where $\hat{\xi}$ maximizes
  the penalized likelihood with the penalty function
  $\tilde{l}(\xi)=\chi(t^2(\xi^1,\ldots,\xi^d))$ satisfying
  \begin{equation}\label{cond0}  
    \chi(t^2)=
    \frac{1}{4}\sum_{i=1}^d
    \int^tS_i\dot{\xi}d\tilde{t}
    -\frac{1}{2}\log\left\{|\varphi'(t^2)|t^d\sqrt{h}\right\}
  \end{equation}
  is $o(n^{-1})$ and coincides with the UMVUE
  up to $O(n^{-1})$ if a complete sufficient statistic
  for $M$ exists. Here, $h$ is the determinant of the Fisher
  metric tensor in the normal coordinate system at $\zeta$,
  $\varphi'(x)=d\varphi/dx$ for $x\in\mathbb{R}_{>0}$,
  $\dot{\xi}=d\xi/dt$, and $\tilde{t}$ is
    the variable of integration.
\end{theorem}

The implication of this theorem becomes clear in the following
example.

\begin{example}
  If the estimand is the squared geodesic distance, i.e.,
  $\varphi$ is the identity map, the right-hand side of
  \eqref{psi_cond} becomes 
  $-d/(4t^2)-d\log h/dt^2/4+\sum_iS_id\xi^i/dt^2/4$, where
  we used the relation between 0 and $(-1)$-Laplacians
  \eqref{al_Lap_rel}, expressions (A.6), (A.7), and
  Lemma~A.3 (ii) in Appendix~A.3.
  Subsequently, the solution to \eqref{psi_cond}
  coincides with that of Theorem~\ref{theo:cond0}.
  The foliation of $M$ by the level surfaces of $f$ is
  reduced to leaves consisting of equidistant points
  from the point
  $\zeta\in M$. For any point $\xi$ in $M$, there exists
  a geodesic $\gamma$ joining $\zeta$ and $\xi$, which is
  transverse to the foliation and diagonally intersects
  the leaves (see Figure~\ref{fig:fig3}).
\end{example}
\begin{figure}
  \centering
  \includegraphics[height=35mm]{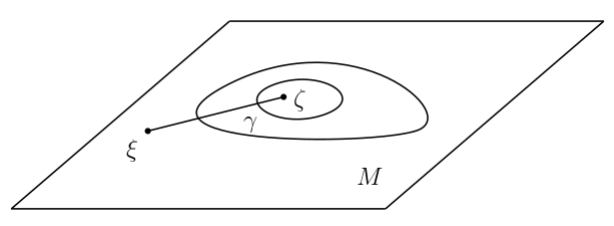}
  \caption{The foliation of a model manifold $M$ by
    the equidistant points from the point $\zeta\in M$.
    A geodesic $\gamma$ joining $\zeta$ and a point
    $\xi\in M$ is also shown.}\label{fig:fig3}
\end{figure}

\section{Integrability and flat model manifolds}
\label{sect:flat}

Let us discuss integration of the partial differential equation
\eqref{pde} for flat model manifolds. After introducing
the notion of the $\alpha$-flatness, we discuss one-dimensional
models in Section~\ref{subs:flat1}, because one-dimensional
$C^\infty$-manifold is always flat
(Proposition~A.4 in Appendix~A.4).
We discuss $\alpha$-flat model manifolds in Section~\ref{subs:flat2}.
We assume the assumptions of Lemma~\ref{lemm:cond}. 

An affine connection on a manifold $M$ is said to be \textit{flat}
if and only if a local coordinate system exists around each point
such that $\Gamma^i_{jk}=0$ for all $i$, $j$, and $k$ \citep{NS94}.
Moreover,

\begin{definition}[\cite{AN00}]
  A manifold with a flat $\alpha$-connection is said to be
  $\alpha$-\textit{flat.} A local coordinate system
  such that ${\Gamma^{(\alpha)}}^i_{jk}=0$ for all $i$, $j$, and $k$
  is called an $\alpha$-\textit{affine coordinate system.}
\end{definition}

\begin{example}\label{exam:a=1}
  An exponential family with canonical parameterization is 1-flat
  because the log-likelihood is $l(\xi;t)=\sum_it_i\xi^i-\psi(\xi)$
  and the Christoffel symbols vanish:
  $\Gamma^{(1)}_{ij,k}=-\partial_i\partial_j\psi\mathbb{E}_\xi[t_k-\partial_k\psi]=0$ for all $i$, $j$, and $k$. Here, $\xi$ comprises a $1$-affine
  coordinate system.
\end{example}  


Suppose we seek a solution $f$ for a system of linear first-order
system of partial differential equations
\begin{equation}\label{lee_pde}
  \partial_i f(\xi)=h_i(\xi), \quad i\in\{1,\ldots,d\},
\end{equation}
where $h_i$ are smooth  functions defined on an open subset
of $\mathbb{R}^d$. If $d>1$, the system is {\it overdetermined,}
indicating that there are more equations than unknown function
$f$. An overdetermined system has a solution only if the partial
differential equations satisfy certain compatibility conditions.
In fact, because $\partial_i\partial_jf=\partial_j\partial_if$
for all $i\neq j$, it is obvious that
\begin{equation}\label{lee_cond}
  \partial_j h_i(\xi)=\partial_i h_j(\xi) \quad
  \text{for~all} \quad i\neq j 
\end{equation}  
is a necessary condition for \eqref{lee_pde} to have a solution
in a neighborhood of any point with an arbitrary initial value.
By virtue of the Frobenius theorem, we can show that
\eqref{lee_cond} is sufficient (Proposition 19.17 of \cite{Lee02}).
The condition of \eqref{lee_cond} is called
the {\it integrability condition} for \eqref{lee_pde}.
\color{black}

Lemma~A.5 in Appendix~A.4 shows that if a model manifold is
$\alpha_0$-flat for non-zero $\alpha_0$, then the system of
partial differential equations $\nabla^{(\alpha)}v=0$ for
a differential form $v$ is integrable for any
$\alpha\in\mathbb{R}$. We call the density of $v$ satisfying
$\nabla^{(\alpha)}v=0$, that is,
\begin{equation}
  v(\alpha;\alpha_0):=g^{\frac{1}{2}-\frac{\alpha}{2\alpha_0}},
  \quad \alpha\in\mathbb{R}
  ~\text{and}~\alpha_0\in\mathbb{R}\setminus \{0\},
\label{al_prior}
\end{equation}
the density of the $\alpha$-{\it parallel volume element
with respect to the $\alpha_0$-flat manifold.} Note that
$v(0;\alpha_0)=\sqrt{g}$ is the Jeffreys prior irrespective
of $\alpha_0$. $\alpha$-parallel volume elements have
appeared in the statistical literature including \cite{TA05}
(Remark A.7 in Appendix A.4).

\subsection{One-dimensional model manifolds}
\label{subs:flat1}

For one-dimensional models, the partial differential equation
\eqref{pde} is reduced to an ordinary differential equation and
can be integrated as in the following corollary of
Theorem~\ref{theo:UMVUE2}. The proof is in Appendix~A.4.

\begin{corollary}\label{coro:1dim}
  Consider a one-dimensional model manifold $M$.
  The bias of the plug-in estimator $f(\hat{\xi})$ of an estimand
  $f(\xi)$, $\xi\in\Xi$ satisfying $f'(\xi)=df/d\xi\neq 0$,
  where $\hat{\xi}$ maximizes the penalized likelihood
  with the penalty function $\tilde{l}(\xi)$ satisfying
  \begin{equation}\label{1dim}
    e^{\tilde{l}(\xi)}\propto\frac{\{g(\xi)\}^{1/4}}
    {\sqrt{|f'(\xi)|}}
    e^{\frac{1}{4}\int^\xi S_1(\tilde{\xi})d\tilde{\xi}},
  \end{equation}
  is $o(n^{-1})$, and coincides with
  the UMVUE up to $O(n^{-1})$ if a complete sufficient
  statistic for $M$ exists. Here,
  $S_1=g^{11}S_{111}$, $g^{11}=g_{11}^{-1}$, and $g=g_{11}$.
  In particular, for an $\alpha$-flat manifold of non-zero
  $\alpha$ $\in\mathbb{R}$ with
  an $\alpha$-affine coordinate system $\xi$,
  the condition becomes
  \begin{equation}\label{1dim_af}
    e^{\tilde{l}(\xi)}\propto
    \frac{v((\alpha-1)/2;\alpha)}{\sqrt{|f'(\xi)|}},
  \end{equation}
  where $v((\alpha-1)/2;\alpha)$ is the density of
  the $(\alpha-1)/2$-parallel volume element with respect
  to an $\alpha$-flat manifold defined in \eqref{al_prior}.
\end{corollary}

\begin{remark}
  Any one-dimensional $C^\infty$-manifold is flat, and a curve in
  the manifold is always a geodesic. Therefore, any
  local coordinate is a normal coordinate, and
  the expression \eqref{cond0} with $d=1$ should be
  consistent with the expression \eqref{1dim}. In fact,
  using the facts that
  $g=g_{11}$ is a constant by
  the canonical parameterization of the geodesic:
  $g_{11}\{\dot{\xi}(0)\}^2=1$ and
  $f'(\xi)=d\varphi(t^2)/dt^2\cdot dt^2/d\xi=2\varphi'(t^2)t/\dot{\xi}(0)$,
  we confirm that \eqref{cond0} yields \eqref{1dim}.
  Moreover, if $\xi$ is an $\alpha$-affine coordinate
  system, \eqref{1dim} reduces to \eqref{1dim_af},
  since $g'=\alpha S_{111}$ by \eqref{cov_met} with
  $\Gamma^{(\alpha)}_{11,1}=0$.
\end{remark}    

\subsection{Flat model manifolds and generalization of Firth's method}
\label{subs:flat2}

We can obtain some explicit results for $\alpha$-flat model manifolds.
Using the relation between $\alpha$ and $(-1)$-Laplacians
\eqref{al_Lap_rel}, we recast \eqref{pde} into
\begin{equation}\label{cond_af1}
  \sum_i({\rm grad}f)^i\partial_i\tilde{l}
  =-\frac{1}{2}\Delta^{(\alpha)}f
  +\frac{1+\alpha}{4}\sum_iS_{i}({\rm grad}f)^i.
\end{equation}
Obviously, a sufficient condition for the penalty function $\tilde{l}$
satisfying \eqref{cond_af1} is the following overdetermined system of 
partial differential equations
\begin{equation}\label{cond_suf}
  \partial_i\tilde{l}=-\frac{1}{2}
  \frac{\Delta^{(\alpha)}f}{({\rm grad}f)^i}
  +\frac{1+\alpha}{4}S_i \quad \text{for~all} \quad
  i\in\{1,\ldots,d\}.
\end{equation}
except for $d=1$. The integration of this system gives the following
corollary of Theorem~\ref{theo:UMVUE2}. The proof is in Appendix A.4,
in which we showed that the integrability condition for
\eqref{cond_suf} is satisfied.

\begin{corollary}\label{coro:alp}
  Consider an $\alpha$-flat model manifold $M$ for non-zero
  $\alpha\in\mathbb{R}$ and an $\alpha$-harmonic estimand
  function $f$. The bias of the plug-in estimator $f(\hat{\xi})$
  of an estimand $f(\xi)$, $\xi\in\Xi$, where $\hat{\xi}$
  maximizes the penalized likelihood with the penalty function
  $\log v((\alpha-1)/2;\alpha)$, is $o(n^{-1})$
  and coincides with the UMVUE up to $O(n^{-1})$ if a complete
  sufficient statistic for $M$ exists. Here,
  $v((\alpha-1)/2;\alpha)$ is the density of
  the $(\alpha-1)/2$-parallel volume element with respect to
  the $\alpha$-flat model manifold defined in \eqref{al_prior},
  and $\xi$ comprises an $\alpha$-affine coordinate system.
\end{corollary}

  
Corollary~\ref{coro:alp} is for estimands of scalar functions,
nevertheless, it gives a generalization of Firth's bias-reduction
method of maximum likelihood estimates of
multidimensional parameters. For an $\alpha$-flat manifold
of non-zero $\alpha$, a component of the $\alpha$-affine
coordinate system $\xi$, say $\xi^{i_0}$, $i_0\in\{1,\ldots,d\}$,
is $\alpha$-harmonic, because
$\Delta^{(\alpha)}\xi^{i_0}=\sum_i\partial^i\partial_i\xi^{i_0}=0$.
Applying Corollary~\ref{coro:alp} with setting $f=\xi^{(i)}$,
we obtain the penalty function $\log v((\alpha-1)/2;\alpha)$
irrespective of $i_0$. Therefore, we have

\begin{corollary}\label{coro:firth}
  For an $\alpha$-flat model manifold of a non-zero $\alpha$
  $\in\mathbb{R}$, the bias of the estimator $\hat{\xi}$ of
  the $\alpha$-affine coordinate $\xi$ that maximizes
  the penalized likelihood with the penalty function
  $\log v((\alpha-1)/2;\alpha)$ is $o(n^{-1})$.
\end{corollary}

The generalized linear model of generic link function is an example
of an $\alpha$-flat model manifold.

\begin{example}
  For generalized linear models of generic link function,
  Theorem~1 of \cite{KF09} says that the integrability condition
  for Firth's method is equivalent to
  $h''_i/h'_i=\omega\kappa_{2i}'/\kappa_{2i}$ for all $i\in\{1,\ldots,d\}$
  and a constant $\omega\in\mathbb{R}$, where $h_i$ is the inverse
  link function, $\kappa_{2i}$ is the variance of the response vector,
  and the derivatives are with respect to the linear predictor.
  They specified the suitable penalty function as
  $v(\omega-1;2\omega-1)$ for $\omega\neq 1/2$. In fact, direct
  computation shows that
  $\Gamma^{(\alpha)}_{rs,t}=\{(\omega-1)+(1-\alpha)/2\}\sum_{i=1}^nh_i^{'2}\kappa_{2i}'/\kappa_{2i}^2x^i_rx^i_sx^i_t$ and it vanishes if $\alpha=2\omega-1$,
  where $x^i_r$ is the design matrix.   Here, we used
  $S_{rst}=\sum_{i=1}^nh_i^{'2}\kappa_{2i}'/\kappa_{2i}^2x^i_rx^i_sx^i_t$,
  which follows by the right equality of \eqref{cov_met} with $\alpha=1$.  
  Further discussion on the canonical link ($\omega=1$) is in
  Section~\ref{subs:exam1}.
\end{example}  

\section{Examples of estimands}
\label{sect:exam}

In this section, several examples are presented with numerical
results to illustrate how the developed bias-reduction method
for generic estimands works for problems of statistical interest.

\subsection{Generalized linear models}
\label{subs:exam1}
Generalized linear models constitute a class of exponential
families that cover both discrete and continuous measurements.
They share several common properties, and the bias reduction
of the maximum likelihood estimate of a parameter has been
actively studied (see Section 15.2 of \cite{MN89} and \cite{Kos14}
for a survey).
Bias reduction of generic estimands is even more challenging.
In this subsection, we apply our bias-reduction method to
estimate the expected responses of generalized linear models.
We will see that the `design dependent shrinkage' property
of our bias reduction works effectively.

For simplicity, we concentrate on a canonical link and a known
dispersion $\phi$. A vector of observations $y$ with $n$
components is assumed to be the realization of a random vector
whose components are independently distributed with mean vector
$\mu$. The log-likelihood of a generalized linear model is
expressed as
\begin{equation}\label{like_glm}
  l(\xi;y)=\frac{\sum_{i=1}^ny_i\xi^i-\psi(\xi)}{\phi}+c(y,\phi)
\end{equation}
for some specific functions $c$ and $\psi$. The linear predictor
$\xi^i$ is modeled by a linear combination of unknown parameters
$\{\beta^1,\ldots,\beta^d\}$:
\begin{equation}\label{design_glm}
  \xi^i=\sum_{r=1}^d{x^i}_r\beta^r, \quad i\in\{1,\ldots,n\},
\end{equation}
where ${x^i}_r$ is the design matrix. By substituting
\eqref{design_glm} into \eqref{like_glm}, it can be observed
that the model is an exponential family with canonical parameter
$\{\beta^1,\ldots,\beta^d\}$ and sufficient statistics
$t_r=\sum_{i=1}^ny_i{x^i}_r$, $r\in\{1,\ldots,d\}$. The score functions
are
\[
  u_r=\partial_r l=
  \sum_{i=1}^n
  \frac{\partial l}{\partial\xi^i}\frac{\partial\xi^i}{\partial\beta^r}
  =\frac{1}{\phi}\sum_{i=1}^n(y_i-\mu_i){x^i}_r, \quad
  \partial_r:=\frac{\partial}{\partial\beta^r}, \quad
  r\in\{1,\ldots,d\},
\]
with the means
$\mu_i=\mathbb{E}_\beta[y]=\partial{\psi}/\partial\xi^i$.
The {\it link function} of a generalized linear model relates
the linear predictors $\xi^i$ to the means $\mu_i$. We assume
that the link function is invertible, and let the inverse link be
denoted by $h$, that is, $\mu_i=h(\xi^i)$. The Fisher metric
tensor is given by
\[
  g_{rs}=-\partial_r\partial_sl
  =\frac{1}{\phi}\sum_{i=1}^n\partial_s\mu_i{x^i}_r
  =\frac{1}{\phi}\sum_{i=1}^n\frac{d\mu_i}{d\xi^i}\partial_s\xi^i{x^i}_r
  =\frac{1}{\phi}\sum_{i=1}^nh'(\xi^i){x^i}_r{x^i}_s.
\] 
It is convenient to use the following representation in
matrices:
\begin{equation}\label{metric_glm}
  (g_{rs})=\frac{1}{\phi}\sum_{i=1}^n({x_r}^i)(w_{ii})({x^i}_s)
  =\frac{1}{\phi}X^TWX,
\end{equation}
where $W={\rm diag}(w_{11},\ldots,w_{nn})$ and
$w_{ii}=h'(\xi^i)$. Since the model is an exponential family
with canonical parameterization, the Christoffel symbols are
$\Gamma^{(1)}_{rs,t}=0$ for all  $r$, $s$, and $t$
(Example~\ref{exam:a=1}), and we have
\begin{equation}\label{con_glm}
  \Gamma_{rs,t}^{(-1)}=S_{rst}=\mathbb{E}_\beta[u_ru_su_t]
  =\sum_{i=1}^n
  \frac{\kappa_{3i}}{\phi^3}{x^i}_r{x^i}_s{x^i}_t,
\end{equation}
where $\kappa_{3i}=\mathbb{E}_\beta(y_i-\mu_i)^3$ is the third
cumulant of the response vector and the first
  equality follows by \eqref{al_con}.

To discuss the expected response, we choose the inverse
link function $h(\xi^{i_0})$ for the $i_0$-th design as
the estimand. The condition \eqref{pde} for the design-specific
penalty function $\tilde{l}^{(i_0)}$ is
\begin{equation}\label{pde_glm}
  \langle{\rm grad} h(\xi^{i_0}),{\rm grad}\tilde{l}^{(i_0)}\rangle
  +\frac{1}{2}\Delta^{(-1)}h(\xi^{i_0})=0.
\end{equation}
Using the matrix representation \eqref{metric_glm}, we obtain
\[
  \Delta^{(1)}h(\xi^{i_0})
  =\sum_{r,s}g^{rs}\partial_r\partial_sh(\xi^{i_0})
  =\phi\mathbf{x}_{i_0}(X^\top WX)^{-1}\mathbf{x}_{i_0}^\top
  h''(\xi^{i_0}),
\]
where $\mathbf{x}_{i_0}$ is the $i_0$-th row vector of $X$.
It is evident that a linear link function is 1-harmonic.
According to Corollary~\ref{coro:alp}, our method reduces
the bias to be $o(n^{-1})$ by the Jeffreys prior.
Incidentally, this conclusion is the same as that for
the estimation of parameter $\beta$ discussed by Firth
\cite{Fir93}. Otherwise, by using the relation between
$\alpha$-Laplacians \eqref{al_Lap_rel} and the expression
\eqref{con_glm}, we have
\begin{align*}
  \Delta^{(-1)}h(\xi^{i_0})
  =&\phi\mathbf{x}_{i_0}(X^\top WX)^{-1}\mathbf{x}_{i_0}^\top
  h''(\xi^{i_0})\\
  &-\sum_{j=1}^n\frac{\kappa_{3j}}{\phi}
  \mathbf{x}_j(X^\top WX)^{-1}\mathbf{x}_{i_0}^\top
  \mathbf{x}_j(X^\top WX)^{-1}\mathbf{x}_j^\top h'(\xi^{i_0}).
\end{align*}
If this expression is identically zero, or $h(\xi^{i_0})$
is $(-1)$-harmonic, asymptotic unbiasedness can be
achieved without a penalty (Corollary~\ref{coro:no_adjm}).
If the inverse link function is neither linear nor
$(-1)$-harmonic, we have to solve the partial differential equation
\eqref{pde_glm}. Noting the expression
\[
  |{\rm grad}h(\xi^{i_0})|=\phi \mathbf{x}_{i_0}(X^\top WX)^{-1}
  \mathbf{x}_{i_0}^\top\{h'(\xi^{i_0})\}^2,
\]
the differential equation \eqref{psi_cond} takes the form
\[
  \frac{d\chi}{dh}=-\frac{1}{2}\frac{h''}{h'^2}+\frac{1}{2\phi^2}
  \frac{\sum_{j=1}^n\kappa_{3j}
  \mathbf{x}_j(X^\top WX)^{-1}\mathbf{x}_{i_0}^\top
  \mathbf{x}_j(X^\top WX)^{-1}\mathbf{x}_j^\top}
  {\mathbf{x}_{i_0}(X^\top WX)^{-1}\mathbf{x}_{i_0}^\top h'}.
\]
Multiplying the right-hand side by $h'(\xi^{i_0})$ and integrating
with respect to $\xi^{i_0}$ yields
\begin{equation}\label{pen_glm}
  -\frac{1}{2}\log|h'(\xi^{i_0})|
  +\frac{1}{2\phi^2}
  \int^{\xi^{i_0}}
  \frac{\sum_{j=1}^n\kappa_{3j}\mathbf{x}_j(X^\top WX)^{-1}\mathbf{x}_{i_0}^\top\mathbf{x}_j(X^\top WX)^{-1}\mathbf{x}_j^\top}
  {\mathbf{x}_{i_0}(X^\top WX)^{-1}\mathbf{x}_{i_0}^\top}d\tilde{\xi}.
\end{equation}
If \eqref{pen_glm} is reduced to an injective function $\psi$ of
$h(\xi^{i_0})$, then it becomes the desired penalty function
$\tilde{l}^{(i_0)}(\beta^1,\ldots,\beta^d)
=\chi(h(\xi^{i_0}))=\chi(h({x^{i_0}}_r\beta^r))$.
The penalized likelihood can be maximized with Fisher's scoring
on the denormalized model (see Remark~\ref{rema:fisher}).
By the Fisher metric tensor of the denormalized model manifold
$\tilde{g}_{rs}:=g_{rs}-\partial_r\partial_s\tilde{l}^{(i_0)}$,
the updated rule for parameter $\beta$ is represented as
\begin{equation}\label{scoring}
  \beta^r\leftarrow
  \sum_{s=1}^d\tilde{g}^{rs}(\sum_{t=1}^d\tilde{g}_{st}\beta^t+u_s+\partial_s\tilde{l}^{(i_0)}).
\end{equation}
  
Bias reduction of maximum likelihood estimates of regression
coefficients, or parameters, in logistic regressions has been
actively studied, including \cite{Cop88} and \cite{Fir93}.
In logistic regression, the expected response is the event
frequency of each design, which is given by the logistic
function of the linear predictor
$\pi_i=e^{\xi^i}/(1+e^{\xi^i})$ for the $i$-th design.
The maximum likelihood estimate of parameter $\beta$ tends
to be biased away from the origin $\beta=0$. The `separation'
phenomenon, in which a parameter estimate diverges to infinity,
causes instability of the logistic regression. This phenomenon
occurs when the event occurs or not is determined by whether
the linear predictor is positive or negative in a sample.
Firth's bias-reduction method has been a tool to avoid
this separation phenomenon \cite{Kos14}. Therefore, a bias
reduction requires some degree of `shrinkage' of an estimate
toward the origin.

In logistic regression, we are also interested in
the event frequency. We can estimate it by the plug-in
estimator with the estimates of $\beta$, but removing
the bias of $\beta$ introduces bias in the plug-in
estimator. Since the logistic regression is
an exponential family and the linear predictor is
the canonical parameter, the penalty function of Firth
is $\log g/2$, where $g$ is the determinant of the Fisher
metric tensor (see Introduction). 
Elgmati's \cite{Elg15} proposal uses $\lambda\log g$ as
the penalty function, where $\lambda\in[0,1/2]$ is
a tuning parameter. In practice, we do not know the optimal
$\lambda$. Elgmati et al.'s proposed choices of
$\lambda$ for the case of two design points such that
it minimizes the mean squared error under some a priori
assumptions.

As an illustration, we discuss a variant of the univariate
logistic regression used by Copus \cite{Cop88} with the severe
setting considered by Firth \cite{Fir93}; that is, one
observation is made at each of the five design points with
$\xi^i={x^i}_1\beta^1=i\beta$ for $i=0,\pm1,\pm2$, 
where the indexing is changed from \eqref{design_glm}
  such that the number indicates the design point of each observation.
The Fisher metric tensor \eqref{metric_glm} is
\[
  g_{11}=4\pi_{-2}(1-\pi_{-2})+
  \pi_{-1}(1-\pi_{-1})+\pi_{1}(1-\pi_{1})+4\pi_{2}(1-\pi_{2}).
\]
The penalty function can be obtained by calculating
\eqref{pen_glm}; however, this model is one-dimensional, and
the application of Corollary~\ref{coro:1dim} is easier.
Given that the 0-parallel volume element is the Jeffreys prior,
we immediately obtain the penalty function
\[
  \tilde{l}^{(i_0)}(\beta)=
  \frac{1-i_0}{2}\beta+\log(1+e^{i_0\beta})
  -\log(1+e^\beta)-\log(1+e^{2\beta})
  +\frac{1}{2}\log \{(1+e^\beta)^4+4e^{2\beta}\}.
\]

The sufficient statistic $t_1=\sum_{i=-2}^2y_i{x^i}_1=\sum_{i=-2}^2iy_i$
has seven possible values. Table~1
is an extension of Table 1 in \cite{Fir93}. It displays
the distribution of the estimates of parameter $\beta$ for
each value by the three methods: the maximum likelihood
estimation (MLE), the maximizer of the penalized likelihood
proposed by Firth (Firth), and ours (asymptotic unbiased
estimator, abbreviated as AUE).
The AUE of $\beta$ is design-specific since it is chosen to reduce
the bias of the estimate of the event frequency of each design.
The limiting behaviors with $\beta\to\pm\infty$ show that
the maximum likelihood estimate $\hat{\beta}_{\rm MLE}$ does not
exist if $t_1=\pm 3$ and the AUE of $\beta$ for the event frequencies
$\pi_{\pm 1}$ or $\pi_{\pm 2}$ does not exist if $t_1=\pm 3$.

The results were obtained with Fisher's scoring, and the iterations
were continued until the absolute value of the update of
the estimate became smaller than $10^{-5}$. The average number
of required iterations was less than 50. Comparing the results
of Firth with AUE, it can be noted that in AUE, shrinkage is
weak, or even `anti-shrinkage' occurs for large absolute values
of the linear predictor. This is reasonable because estimating
responses becomes easier with an increase in the absolute
value of the predictor.

Table~2 summarizes the performance of
the estimators of the event frequencies.
The maximum likelihood estimates are in the rows of MLE.
The plug-in estimates by the maximizer of Firth's penalized
likelihood and Elgmati et al's extension of
  Firth's method are in the rows of Firth and Elgmati, respectively,
and those of our penalized likelihood are in the rows of AUE.
The mean squared errors are shown in the columns of MSE.
In the Elgmati, we chose tuning parameter
$\lambda$, which minimized the mean squared errors
within $\{0,0.1,0.2,0.3,0.4,0.5\}$. For $\beta=0.5$ and 1,
$\lambda=0.5$ (Firth) was optimal, so we omitted the rows
for the Elgmati.
\color{black}
The results showed that AUE has a significantly smaller
bias than MLE. The mean squared errors were similar in
magnitude. This finding is consistent with the second-order
efficiency of both estimators (Remark~\ref{rema:mse}).
The bias of Firth increased with the increases in
$\beta$, which comes from the non-linearity of the logistic
function.
Elgmati improved the bias but still had a bias for
large absolute values of the linear predictor, even with
the optimal $\lambda$.

\begin{table}[t]
  \small
  \caption{Distribution of the estimates of the parameter
    in a logistic regression model.}
  \label{table:tab1}
  \centering
  \begin{tabular}{rrrrrrrrr}
  & & & \multicolumn{3}{c}{AUE} & \multicolumn{3}{c}{Sampling probabilities} \\
  $t_1$  & MLE & Firth & $\pi_{\pm2}$ & $\pi_{\pm1}$ & $\pi_0$ & $\beta=0.5$ & $\beta=1$ & $\beta=1.5$\\
  \hline
  $-3$ & $-\infty$&$-1.383$ &$-\infty$& $-\infty$  & $-1.383$ & 0.010 & 0.001 & 0.000 \\
  $-2$ & $-1.012$  &$-0.683$ & $-1.205$ & $-0.771$ & $-0.683$ & 0.034 & 0.006 & 0.001 \\
  $-1$ & $-0.420$  &$-0.307$ & $-0.452$ & $-0.335$ & $-0.307$ & 0.084 & 0.023 & 0.005 \\
  $0$  & 0        & 0      & 0       & 0       & 0       & 0.185 & 0.083 & 0.027 \\ 
  $1$  &  $0.420$  & $0.307$ &  $0.452$ &  $0.335$ &  $0.307$ & 0.229 & 0.168 & 0.091 \\
  $2$  &  $1.012$  & $0.683$ &  $1.205$ &  $0.771$ &  $0.683$ & 0.251 & 0.305 & 0.271 \\
  $3$  & $+\infty$ & $1.383$ & $+\infty$& $+\infty$&  $1.383$ & 0.207 & 0.415 & 0.607 \\
  \hline
  \end{tabular}
\end{table}

\begin{table}[t]
  \caption{Biases and mean squared errors of the estimators of
    the success probabilities in a logistic regression model.}
  \label{table:tab2}
\small
\centering
$\beta=0.5$\\
\begin{tabular}{rrrrrrrrrrr}
  &\multicolumn{2}{c}{$\pi_{-2}=0.269$}&\multicolumn{2}{c}{$\pi_{-1}=0.378$}
  &\multicolumn{2}{c}{$\pi_0=0.5$}     &\multicolumn{2}{c}{$\pi_1=0.622$}
  &\multicolumn{2}{c}{$\pi_2=0.731$}\\
     &Bias   &MSE  &Bias    &MSE  &Bias &MSE  &Bias    &MSE  &Bias    &MSE\\
\hline  
MLE  &$0.021$&0.065&$-0.042$&0.048&0.098&0.054& $0.042$&0.048&$-0.021$&0.065\\
Firth&$0.058$&0.048& $0.017$&0.018&0    &0    &$-0.017$&0.018&$-0.058$&0.048\\
AUE  &$0.011$&0.070&$-0.028$&0.044&0    &0    & $0.028$&0.044&$-0.011$&0.070\\
\hline
\end{tabular}

\smallskip

$\beta=1$\\
\begin{tabular}{rrrrrrrrrrr}
  &\multicolumn{2}{c}{$\pi_{-2}=0.119$}&\multicolumn{2}{c}{$\pi_{-1}=0.269$}
  &\multicolumn{2}{c}{$\pi_0=0.5$}     &\multicolumn{2}{c}{$\pi_1=0.731$}
  &\multicolumn{2}{c}{$\pi_2=0.881$}\\
     &Bias   &MSE  &Bias    &MSE  &Bias &MSE  &Bias    &MSE  &Bias    &MSE\\
\hline  
MLE  &$0.030$&0.035&$-0.061$&0.041&0.207&0.104& $0.061$&0.041&$-0.030$&0.035\\
Firth&$0.088$&0.034& $0.047$&0.015&0    &0    &$-0.047$&0.015&$-0.088$&0.034\\
AUE  &$0.018$&0.035&$-0.043$&0.043&0    &0    & $0.043$&0.043&$-0.018$&0.035\\
\hline
\end{tabular}

\smallskip

$\beta=1.5$\\
\begin{tabular}{rrrrrrrrrrr}
  &\multicolumn{2}{c}{$\pi_{-2}=0.047$}&\multicolumn{2}{c}{$\pi_{-1}=0.182$}
  &\multicolumn{2}{c}{$\pi_0=0.5$}     &\multicolumn{2}{c}{$\pi_1=0.818$}
  &\multicolumn{2}{c}{$\pi_2=0.953$}\\
     &Bias   &MSE  &Bias    &MSE  &Bias &MSE  &Bias    &MSE  &Bias    &MSE  \\
\hline  
MLE  &$0.029$&0.016&$-0.058$&0.030&0.303&0.152& $0.058$&0.030&$-0.029$&0.016\\
Firth&$0.092$&0.023& $0.085$&0.015&0    &0    &$-0.085$&0.015&$-0.092$&0.023\\
Elgmati&$0.062$&0.020& $0.040$&0.014&0  &0    &$-0.040$&0.014&$-0.062$&0.020\\
AUE  &$0.018$&0.015&$-0.042$&0.034&0    &0    & $0.042$&0.034&$-0.018$&0.015\\
\hline
\end{tabular}
\end{table}

\subsection{Linear mixed-effects models}
\label{subs:exam2}

Linear mixed-effects models are among the most widely used
classes of statistical models (see Section 3.4 of \cite{LC98}).
In this subsection, we explain how our bias reduction works for
the estimation of the `shrinkage factor' in linear mixed-effects
models.

Efron and Morris \cite{EM72} discussed an estimation of the mean
vector $z$ in $\mathbb{R}^n$, $n\ge 3$ of a linear mixed-effects
model with unknown variance $\sigma^2>0$:
\begin{equation}\label{efron}
  x_i|z_i \sim \text{N}(z_i,1), \quad
  z_i \overset{{\rm iid}}{\sim} \text{N}(0,\sigma^2),
  \quad i\in\{1,\ldots,n\}.
\end{equation}
If we regard the normal distribution $\text{N}(0,\sigma^2)$ as
the prior distribution for $z_i$, $i\in\{1,\ldots,n\}$,
the marginal log-likelihood is
\begin{equation}\label{like_efron}
  l(\sigma^2;x)
  =-\frac{|x|^2}{2(1+\sigma^2)}
   -\frac{n}{2}\log\left\{2\pi(1+\sigma^2)\right\}
\end{equation}
which is obtained by integrating out $z$ from \eqref{efron}.
The estimator of $\sigma^2$ which maximizes \eqref{like_efron}
is $\hat{\sigma^2}_{\rm MLE}=|x|^2/n-1$. The best predictor of
the mean vector $z$ under the squared error loss is (see
Section 3.5 of \cite{LC98})
$\hat{z}(\sigma^2)=\{1-s(\sigma^2)\}x$,
where $s(\sigma^2):=1/(1+\sigma^2)$
is called the {\it shrinkage factor},
because the predictor $\hat{z}$ is shrunken toward the ground
mean, 0, as the variance $\sigma^2$ decreases. In applications
of the model we are interested in the shrinkage factor.
The plug-in estimator $s(\hat{\sigma}_{\rm MLE}^2)$ has bias,
$2/\{(1+\sigma^2)n\}+o(n^{-1})$, and our bias reduction can be
applied to remove this bias as follows.

The log-likelihood \eqref{like_efron} is a one-dimensional
exponential family with the canonical parameter
$\xi=-1/\{2(1+\sigma^2)\}$. The manifold is 1-flat,
and the canonical parameter is a 1-affine coordinate.
The Fisher metric tensor is $g_{11}=n/(2\xi^2)$. As $|x|^2$
is a complete sufficient statistic, Corollary~\ref{coro:1dim}
immediately provides the UMVUE of the shrinkage factor up to
$O(n^{-1})$. By substituting $f(\xi)=s(\sigma^2)=-2\xi$ into
\eqref{1dim_af}, we obtain
$\tilde{l}(\xi)=\log h(0;1)+{\rm const.}=\log(1+\sigma^2)$,
which coincides with the Jeffreys prior.

A two-parameter family of linear mixed-effects models that
generalizes the model expressed by \eqref{efron} is
\[
  x_{ij}|z_i\overset{\text{iid}}{\sim}\text{N}(z_i,\delta),
  \quad j\in\{1,...,m_i\},\quad
  z_i\overset{\text{iid}}{\sim}\text{N}(0,\alpha),
  \quad i\in\{1,...,n\},
\]
where the variances $\alpha>0$ and $\delta>0$ are unknown.
The parameter is denoted by $\xi=(\alpha,\delta)$.
This model is
a simplified version of the linear regression model discussed
in \cite{BHF88}. Let us consider the asymptotics $n\to\infty$
under the assumption that $\sup_i m_i<\infty$. We exclude
the case of $m_1=\cdots=m_n=1$, which reduces to
\eqref{efron}. The best predictor of each mean $z_i$ is
\[
  \hat{z}_i=\{1-s^{(i)}(\xi)\}\bar{x}_i, \quad
  s^{(i)}(\xi)
  :=\frac{\delta}{\delta+m_i\alpha}, \quad
  \bar{x}_i=\frac{\sum_{j=1}^{m_i}x_{ij}}{m_i}
\]
for $i\in\{1,\ldots,n\}$. The log-likelihood is
\[
  l(\xi;x)=\frac{1}{2\delta}
  \left\{
  \sum_{i=1}^n\frac{(m_i\bar{x}_i)^2\alpha}{\delta+m_i\alpha}
  -\sum_{i=1}^n\sum_{j=1}^{m_i}x_{ij}^2\right\}
  -\frac{1}{2}\sum_{i=1}^n
  \log\left\{(\delta+m_i\alpha)\delta^{m_i-1}\right\}+{\rm const.}
\]

The model manifold is the orthant
$\{\xi=(\alpha,\delta):\xi\in\mathbb{R}^2_{>0}\}$.
The Fisher metric tensor is \citep{HL21}
\[
  (g_{ij})=\frac{1}{2}\left(
  \begin{array}{cc}
  \sum_{i=1}^{n}m_i^2(\delta+m_i\alpha)^{-2}&
  \sum_{i=1}^{n}m_i(\delta+m_i\alpha)^{-2}\\
  \sum_{i=1}^{n}m_i(\delta+m_i\alpha)^{-2}&
  \sum_{i=1}^{n}(\delta+m_i\alpha)^{-2}+(m-n)\delta^{-2}
  \end{array}
  \right)
\]
with the determinant 
\begin{align*}
  g&=\frac{1}{4}
  \left\{\sum_{i<j}\frac{(m_i-m_j)^2}{(\delta+m_i\alpha)^2
  (\delta+m_j\alpha)^2}
  +\frac{m-n}{\delta^2}\sum_{i=1}^{n}\frac{m_i^2}{(\delta+m_i\alpha)^2}
  \right\}>0,
\end{align*}
where $m:=\sum_{i=1}^n m_i$. After some calculations,
we obtain the skewness tensor:
\begin{align*}
  &S_{\alpha\alpha\alpha}
   =\sum_{i=1}^{n}\frac{m_i^3}{(\delta+m_i\alpha)^3},\quad
   S_{\alpha\alpha\delta}=\sum_{i=1}^{n}\frac{m_i^2}{(\delta+m_i\alpha)^3},\\
  &S_{\alpha\delta\delta}=\sum_{i=1}^{n}\frac{m_i}{(\delta+m_i\alpha)^3},\quad
   S_{\delta\delta\delta}=\sum_{i=1}^{n}\frac{1}{(\delta+m_i\alpha)^3}
   +\frac{m-n}{\delta^3}.
\end{align*}
We can confirm that $\Gamma^{(-1)}_{ij,k}=0$ for all $i$, $j$,
and $k$. Therefore, the model manifold is $(-1)$-flat, and $\xi$
is a $(-1)$-affine coordinate system. Corollary~\ref{coro:firth}
implies that the maximum likelihood estimate $\hat{\xi}_{\rm MLE}$
of parameter $\xi$ is second-order unbiased without penalty.

The bias reduction of the shrinkage factor $s^{(i)}$ is
as follows. The condition \eqref{pde} for the penalty function
$\tilde{l}^{(i)}$ becomes
$\langle{\rm grad} s^{(i)},{\rm grad}\tilde{l}^{(i)}\rangle
+\Delta^{(-1)}s^{(i)}/2=0$.
Since
\[
  -\frac{1}{2}\frac{\Delta^{(-1)}s^{(i)}}{|{\rm grad}s^{(i)}|}
  =\frac{1}{s^{(i)}}\left(\frac{n}{m}-1\right),
\]
the differential equation \eqref{psi_cond} can be integrated
immediately and we have
\[
  \tilde{l}^{(i)}(\xi)=\chi(s^{(i)})+{\rm const.}
  =\left(1-\frac{n}{m}\right)
  \log\left(1+\frac{m_i\alpha}{\delta}\right), \quad
  m=\sum_{i=1}^nm_i.
\]

Table~3 summarizes the performance of
the estimators of the shrinkage factor $s^{(1)}$.
The results for the plug-in estimator with the maximum likelihood
estimate of the parameter $\xi=(\alpha,\delta)$ are shown in
the row of MLE, wheres those with the maximizer of the penalized
likelihood are shown in the row of AUE. The mean squared errors are
in the columns of MSE. We set $n=50$ and $m_1=\cdots=m_{50}=10$.
The results were obtained from 10,000 experiments. The estimates
of parameter $\xi$ were computed with Fisher's scoring, as described
in Section~\ref{subs:exam1}. The average number of iterations until
the absolute value of the update of the estimate of the shrinkage
factor became smaller than $10^{-5}$ was equal to or less than seven.
The results showed that AUE had a significantly smaller bias
than MLE. The MSEs were in similar magnitude.
\begin{table}
  \caption{Biases and mean squared errors of the estimators
    of the shrinkage factor in a linear mixed-effects model.}
  \label{table:tab3}
  \small
  \centering

  \begin{tabular}{crrrrrr}
  &\multicolumn{2}{c}{$(\alpha,\delta)=(1,1)$}&\multicolumn{2}{c}{$(\alpha,\delta)=(1,5)$}&\multicolumn{2}{c}{$(\alpha,\delta)=(1,10)$}\\
  &\multicolumn{2}{c}{$s^{(1)}=0.09090$}&\multicolumn{2}{c}{$s^{(1)}=0.33333$}&\multicolumn{2}{c}{$s^{(1)}=0.50000$}\\
  &\multicolumn{1}{c}{Bias}&\multicolumn{1}{c}{MSE}&\multicolumn{1}{c}{Bias}&\multicolumn{1}{c}{MSE}&\multicolumn{1}{c}{Bias}&\multicolumn{1}{c}{MSE}\\
  \hline
  MLE&$ 0.00370$&$0.00044$&$ 0.01387$&$0.00602$&$ 0.02081$&$0.01337$\\
  AUE      &$-0.00007$&$0.00040$&$ 0.00004$&$0.00537$&$ 0.00011$&$0.01198$\\
  \hline
\end{tabular}

\end{table}

\subsection{Location-scale family and hyperbolic space}
\label{subs:exam3}

The location-scale family is a group family
whose group of transformations is
$x\mapsto \sigma x+\mu \in {\rm PSL}(2,\mathbb{R})$ for
$x\in\mathbb{R}$, where $\mu\in\mathbb{R}$ and $\sigma>0$
are the location and scale parameter, respectively, and
${\rm PSL}(2,\mathbb{R})$ denotes the projective special linear
group in $\mathbb{R}^2$ (see Sections 3.3 and 4.4 of \cite{LC98}).
From a geometrical perspective, the family is represented
as a two-dimensional hyperbolic space and is not $\alpha$-flat
for any $\alpha$. The purpose of this subsection is to
illustrate how our bias-reduction method works for non-flat
manifolds. The Neyman--Scott model, in which the number of
parameters increases with the number of observations and
the maximum likelihood estimator is inconsistent, can
also be discussed in a similar manner, because the model
manifold is also a hyperbolic space.

First, we explain why the location-scale family is not flat.
Let $p_0(z)$, $z\in\mathbb{R}$ is a probability density
symmetric about the origin. The location-scale family of
the standard density $p_0$ is given by the density
\begin{equation}\label{ls}
  p(x)dx=\frac{1}{\sigma}p_0\left(\frac{x-\mu}{\sigma}\right)dx,
\end{equation}
where the parameter is denoted by $\xi=(\mu,\sigma)$.
The one-sample log-likelihood is
$l(\xi;x)=l_0\{(x-\mu)/\sigma\}-\log\sigma$,
where $l_0(z)=\log p_0(z)$, and we have
\begin{align*}
  &\frac{\partial l}{\partial\mu}=-\frac{l'_0}{\sigma}, \quad
  \frac{\partial l}{\partial\sigma}=-\frac{zl'_0}{\sigma}-\frac{1}{\sigma},
  \quad
  \frac{\partial^2 l}{\partial\mu^2}=\frac{l''_0}{\sigma^2},\quad
  \frac{\partial^2 l}{\partial\sigma^2}
  =\frac{2zl'_0}{\sigma^2}+\frac{z^2l''_0}{\sigma^2}+\frac{1}{\sigma^2},\\
  &\frac{\partial^2 l}{\partial\sigma\partial\mu}
  =\frac{l'_0}{\sigma^2}+\frac{zl''_0}{\sigma^2}.
\end{align*}
It can be seen that
$\mathbb{E}_\xi[\partial_\mu l]=\mathbb{E}_\xi[\partial_\sigma l]=\mathbb{E}_\xi[\partial_\mu\partial_\sigma l]=0$,
$\mathbb{E}_\xi(-\partial^2_\mu l)=-\mathbb{E}(l_0'')\sigma^{-2}$, and
$\mathbb{E}_\xi(-\partial^2_\sigma l)={1-\mathbb{E}(z^2l_0'')}\sigma^{-2}$,
where expectations $\mathbb{E}_\xi$ are taken with respect to
the density of \eqref{ls}. For simplicity, we assume the standard
density $p_0(z)$ satisfies the condition
$\mathbb{E}(l''_0)=\mathbb{E}(z^2l''_0)-1=:-R^2<0$,
which is a choice of the scale. Then, the Fisher metric tensor
becomes $g_{ij}=(R/\sigma)^2\delta_{ij}$. This metric is known as
the Poincar\'e metric of the upper half-plane model of
the two-dimensional hyperbolic space of the sectional curvature
$(-R^{-2})$. The skewness tensor has components
$S_{\mu\mu\mu}=S_{\mu\sigma\sigma}=0$,
$S_{\mu\mu\sigma}=c_1\sigma^{-3}$, and
$S_{\sigma\sigma\sigma}=c_2\sigma^{-3}$,
where $c_1:=-\mathbb{E}(zl'^3_0+l^{'2}_0)$,
$c_2:=-\mathbb{E}(zl'_0+1)^3$.
The vector derived by the contraction of the skewness tensor
has components $S^\mu=0$, $S^\sigma=c\sigma R^{-4}$, $c:=c_1+c_2$.
The Christoffel symbols are
\begin{align*}
  &\Gamma^{(\alpha)}_{\sigma\sigma,\mu}=\Gamma^{(\alpha)}_{\mu\sigma,\sigma}=
  \Gamma^{(\alpha)}_{\mu\mu,\mu}=0, ~~
  \Gamma^{(\alpha)}_{\sigma\sigma,\sigma}=
  -\frac{R^2}{\sigma^3}-\frac{\alpha}{2}\frac{c_2}{\sigma^3}, ~~
  \Gamma^{(\alpha)}_{\mu\sigma,\mu}=
  -\frac{R^2}{\sigma^3}-\frac{\alpha}{2}\frac{c_1}{\sigma^3},\\
  &\Gamma^{(\alpha)}_{\mu\mu,\sigma}=
  \frac{R^2}{\sigma^3}-\frac{\alpha}{2}\frac{c_1}{\sigma^3},
\end{align*}
and the non-zero components of the Riemann curvature tensor
are 
\begin{align*}
&R^{(\alpha)}_{\mu\sigma\mu\sigma}=-\frac{R^2}{\sigma^4}
\left(1+\frac{\alpha}{2}\frac{c_1}{R^2}\right)
\left(1+\frac{\alpha}{2}\frac{c_1-c_2}{R^2}\right),\\
&R^{(\alpha)}_{\sigma\mu\mu\sigma}=\frac{R^2}{\sigma^4}
\left(1-\frac{\alpha}{2}\frac{c_1}{R^2}\right)
\left(1-\frac{\alpha}{2}\frac{c_1-c_2}{R^2}\right),
\end{align*}
which shows that the location-scale model is not $\alpha$-flat
for any $\alpha$,
  because an affine connection is flat if and only if
  the curvature tensor field vanishes identically on $M$
  (see Theorem II.9.1 of \cite{KN63} and Appendix~A.4).

Section~4.2 of \cite{Fir93} discussed the bias reduction
of the maximum likelihood estimate of the variance of
the normal distribution with parameterization $(\mu,\sigma^2)$.
He solved a system of two partial differential equations
for a single penalty function
and found that the bias-reduced estimator is the unbiased
sample variance. Our bias-reduction method reproduces his
result.

Substituting $f(\xi)=\sigma^2$ into
the condition \eqref{pde} yields the partial differential
equation
\begin{equation*}
  \frac{\partial\tilde{l}}{\partial\sigma}
  =-\left(1-\frac{c}{2R^2}\right)\frac{1}{2\sigma},
\end{equation*}
and a solution is $\tilde{l}(\xi)=-\{1-c/(2R^2)\}\log\sigma/2$.
The system of estimating equations for an $n$-sample is
\[
  \frac{\partial l^*}{\partial\sigma}=
  \frac{2}{\sigma^3}\sum_{i=1}^n(x_i-\mu)^2-\frac{n}{\sigma}
  -\left(1-\frac{c}{2R^2}\right)\frac{1}{2\sigma}=0, \quad
  \frac{\partial l^*}{\partial\mu}=\frac{2}{\sigma^2}\sum_{i=1}^n
  (x_i-\mu)^2=0, 
\]
where $l^*(\xi;x):=l(\xi;x_1,\ldots,x_n)+\tilde{l}(\xi)$.
The solution is $\hat{\mu}=\bar{x}$ and
\begin{equation}\label{var_hyp}
  \frac{\hat{\sigma}^2}{2}=\frac{n}{n+(1-c/(2R^2))/2}s^2, \quad
  s^2:=\frac{1}{n}\sum_{i=1}^n(x_i-\bar{x})^2,
\end{equation}
where $\bar{x}$ and $s^2$ are the sample mean and variance,
respectively. If we choose $p_0(z)=e^{-z^2}/\sqrt{\pi}$,
where $c_1=4$, $c_2=8$, and $R^2=2$, \eqref{var_hyp} becomes
the unbiased sample variance of the normal distribution.

Lauritzen \cite{Lau87} proposed hypothesis testing for
the coefficient of variation $\gamma:=\sigma/\mu$ for
the normal distribution. He called $s/\bar{x}$ the geometric
ancillary test statistic for the hypothesis that $\gamma=\gamma_0$
for certain $\gamma_0$ values. Here, $\bar{x}$ and $s$ are
the maximum likelihood estimates of $\mu$ and $\sigma$,
respectively, and the test statistic has a bias of $O(n^{-1})$.
Our bias-reduction method can be applied to remove
this bias as follows.

The condition of \eqref{pde} for the penalty function
$\tilde{l}$ becomes
$\langle{\rm grad}\tilde{l},{\rm grad}\gamma\rangle+
\Delta^{(-1)}\gamma/2=0$.
Since
\[
-\frac{1}{2}\frac{\Delta^{(-1)}\gamma}{|{\rm grad}\gamma|}
=-\frac{\gamma^2-c/(4R^2)}{\gamma+\gamma^3},
\]
the differential equation \eqref{psi_cond} can be integrated
immediately. If we choose $p_0(z)=e^{-z^2}/\sqrt{\pi}$,
a solution is
\[
  \tilde{l}(\xi)=\chi(\gamma)+{\rm const.}=
  \frac{c}{4R^2}\log\frac{\sigma}{\mu}-
  \frac{1}{2}\left(1+\frac{c}{4R^2}\right)
  \log\left\{1+\left(\frac{\sigma}{\mu}\right)^2\right\},
\]
and the system of estimating equations for
$\hat{\xi}=(\hat{\mu},\hat{\sigma})$ that maximizes
the penalized likelihood is
\begin{equation}\label{est_gamma}
  \mu=\bar{x}
  +\frac{5}{4n}\frac{\sigma^4}{\mu(\mu^2+\sigma^2)}
  -\frac{3}{4n}\frac{\sigma^2}{\mu}, \quad
  \sigma^2=2(\bar{x}-\mu)^2+2s^2
  -\frac{5}{2n}\frac{\sigma^4}{\mu^2+\sigma^2}+\frac{3}{2n}\sigma^2.
\end{equation}

Table~4 summarizes the performance of plug-in
estimators of the coefficient of variation $\gamma$:
the statistic $s/\bar{x}$ and the plug-in estimator
$\hat{\xi}=\hat{\sigma}/\hat{\mu}$. We set $n=100$.
The results were obtained from 10,000 experiments.
The results for the estimator $s/\bar{x}$ are in the row
of MLE, whereas those for the estimator $\hat{\sigma}/\hat{\mu}$
are in the row of AUE. The system \eqref{est_gamma} was
numerically solved by iteratively updating the estimates of
$\mu$ and $\sigma$ through the substitution of the current
estimates into the right-hand side. This is repeated until
the ratio of the estimates converged, which is equivalent to
the gradient descent. The average number of iterations until
the absolute value of the update became smaller than $10^{-5}$
was less than five. The results did not change when $\mu$ and
$\sigma$ were multiplied by a common positive number.
The results indicate that $\hat{\xi}=\hat{\sigma}/\hat{\mu}$
has a significantly smaller bias than $s/\bar{x}$. The MSEs
were of similar magnitudes for $\gamma=0.2$ and 1, but
the MSE of MLE was larger than AUE for $\gamma=2$.

Let us discuss the estimation of of squared geodesic distance
$t^2$ from the standard density, where
$t=\text{dis}((0,1),(\mu,\sigma))$. For each point
$(\mu,\sigma)$, there exists the unique geodesic
joining $(0,1)$ and $(\mu,\sigma)$, where the geodesic distance
$t$ is the length of the geodesic
(see Section~\ref{subs:bias4} and Appendix~A.3).
From a statistical perspective, the squared geodesic distance
is a natural estimand of the deviation from a hypothesized density.

The geodesic equations ((A.5) in Appendix~A.3) are
$\ddot{\mu}-2\dot{\mu}\dot{\sigma}/\sigma=0$ and
$\ddot{\sigma}+(\dot{\mu}^2-\dot{\sigma}^2)/\sigma=0$,
where $\dot{\mu}=d\mu/dt$, $\dot{\sigma}=d\sigma/dt$,
$\ddot{\mu}=d^2\mu/dt^2$, and $\ddot{\sigma}=d^2\sigma/dt^2$.
The solution satisfying the initial condition
$\xi(0)=(\mu(0),\sigma(0))=(0,1)$ and
$\dot{\xi}(0)=(\dot{\mu}(0),\dot{\sigma}(0))$ is
\begin{equation}\label{co_tr}
\mu(t)=\frac{R\dot{\mu}(0)\tanh(t/R)}
        {1-R\dot{\sigma}(0)\tanh(t/R)},\quad
\sigma(t)=\frac{1}
        {\cosh(t/R)-R\dot{\sigma}(0)\sinh(t/R)}.   
\end{equation}
As expected, the geodesic distance is
  $\text{dis}((0,1),(\mu(t_0),\sigma(t_0)))=
  \int_0^{t_0}\frac{R}{\sigma}\sqrt{\dot{\mu}^2+\dot{\sigma}^2}dt
  =t_0$.
The geodesic is a portion of the semicircle on the upper
half-plane (see Figure~\ref{fig:fig4}):
\[
  \left(\mu-\frac{\dot{\sigma}(0)}{\dot{\mu}(0)}\right)^2
  +\sigma^2=\frac{1}{(R\dot{\mu}(0))^2},
  \quad \sigma>0.
\]
\begin{figure}
  \centering
  \includegraphics[height=40mm]{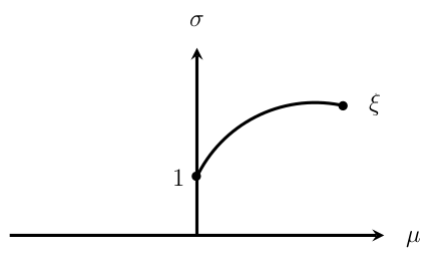}
  \caption{The geodesic distance between a point
    $\xi=(\mu,\sigma)$ and the point $(0,1)$, which corresponds
    to the standard density, is the length of the curve shown
    in the Poincar\'e metric. Our task is to find the estimate
    $\hat{\xi}$ such that $t^2(\hat{\xi})$ is an asymptotically
    unbiased estimate of squared geodesic distance $t^2(\xi)$.}
  \label{fig:fig4}
\end{figure}
The normal coordinate system $\zeta$ at $(0,1)$ is defined as
$\zeta=(\zeta^1,\zeta^2)=(\dot{\mu}(0)t,\dot{\sigma}(0)t)$,
where $(\zeta^1)^2+(\zeta^2)^2=(t/R)^2$.
The relationship \eqref{co_tr} can be regarded as a change of
the basis between $\eta$ and $\xi$, and the inverse is given by
\[
  \zeta^1=\frac{\rho}{\sqrt{\nu^2-1}}\frac{\mu}{\sigma}, \quad
  \zeta^2=\frac{\rho}{\sqrt{\nu^2-1}}\left(\nu-\frac{1}{\sigma}\right),
\]
where $\nu=(\mu^2+\sigma^2+1)/(2\sigma)\ge 1$ and
$\rho=\log(\nu+\sqrt{\nu^2-1})$. The squared geodesic distance
is expressed as a function of $\xi$:
\begin{equation}\label{gdist}
  t^2(\xi)=R^2\rho^2
  =R^2\left[
    \log\left\{
      \frac{\mu^2+\sigma^2+1}{2\sigma}+
      \sqrt{\left(\frac{\mu^2+\sigma^2+1}{2\sigma}\right)^2-1}
      \right\}\right]^2.
\end{equation}
  
Theorem~\ref{theo:cond0} provides us an asymptotic unbiased
estimator of the squared geodesic distance of the bias of
$O(n^{-1})$. The expression \eqref{cond0} contains
the determinant of the Fisher metric tensor in the normal
coordinate system $\eta$.
Therefore, we have to calculate it and then express it in
the original coordinate system $\xi$. By using the transformation
rule of components of a tensor under the change of basis (see
Section I.2 of \cite{KN63}):
$h_{ij}(\eta)=(\partial\xi^k/\partial\eta^i)(\partial\xi^l/\partial\eta^j)g_{kl}(\xi)$,
we obtain
\begin{align*}
  h(\eta(\xi))=&
  \left(\frac{R\mu(\nu\sigma-1)}{\sigma^2(\nu^2-1)}\right)^4
  \left(1+
  \frac{\mu^2}{(\nu\sigma-1)^2}+
  \frac{\sigma^2(\nu^2-1)^2}{\rho^2\mu^2}\right)\\
  &\times\left(1+
  \frac{(\nu\sigma-1)^2}{\mu^2}+
  \frac{\sigma^2(\nu^2-1)^2}{\rho^2(\nu\sigma-1)^2}\right).
\end{align*}
The penalty function is
\begin{equation}\label{pen_hyp}
  \tilde{l}(\xi)=\frac{c}{4R^2}\log\sigma-\log\rho-\frac{1}{4}
  \log h(\eta(\xi)).
\end{equation}
It is straightforward to obtain the system of estimating
equations for $\hat{\xi}$ that maximizes the penalized
likelihood. The explicit expressions are provided in
Appendix~B.

Table~5 summarizes the performance of
the estimators of the squared geodesic distance:
the one is by substituting $(\bar{x},s)$ for $(\mu,\sigma)$
(MLE), and the other is by substituting $(\hat{\mu},\hat{\sigma})$
(AUE), where $(\hat{\mu},\hat{\sigma})$ is the maximizer of the
penalized likelihood. The density was specified to be
$p(z)=e^{-z^2}{\sqrt{\pi}}$. We set $n=100$, and the results
were based on 10,000 experiments. The average number of
iterations until the absolute value of the update of the estimate
of the geodesic distance became smaller than $10^{-5}$ was less
than of equal to four. The results showed that AUE had
a significantly smaller bias than MLE. The MSEs were similar
in magnitude.
It might be surprising that we can reduce the bias of the maximum
likelihood estimate for such a complicated estimand function as
in \eqref{gdist}.

\begin{table}
  \caption{Biases and mean squared errors of the estimators of
    the coefficients of variation of the location-scale family.}
  \label{table:tab4}
\small
\centering
\begin{tabular}{crrrrrr}
  &\multicolumn{2}{c}{$\gamma=0.2$}&\multicolumn{2}{c}{$\gamma=1$}&\multicolumn{2}{c}{$\gamma=2$}\\
  &\multicolumn{1}{c}{bias}&\multicolumn{1}{c}{MSE}&\multicolumn{1}{c}{bias}&\multicolumn{1}{c}{MSE}&\multicolumn{1}{c}{bias}&\multicolumn{1}{c}{MSE}\\
  \hline
  MLE&$-0.00294$&$0.00042$&$-0.00461$&$0.02080$&$ 0.06360$&$0.30433$\\
  AUE&$-0.00000$&$0.00042$&$-0.00023$&$0.02006$&$-0.00301$&$0.20801$\\
  \hline
\end{tabular}
\end{table}

\begin{table}
\small
\centering
\caption{Biases and mean squared errors of the estimators of
  the geodesic distance.}
\label{table:tab5}
\begin{tabular}{crrrrrr}
  &\multicolumn{2}{c}{$(\mu,\sigma)=(1,1)$}&\multicolumn{2}{c}{$(\mu,\sigma)=(5,1)$}&\multicolumn{2}{c}{$(\mu,\sigma)=(0,0.1)$}\\
  &\multicolumn{2}{c}{$t^2=1.85252$}&\multicolumn{2}{c}{$t^2=21.70696$}&\multicolumn{2}{c}{$t^2=10.60380$}\\
  &\multicolumn{1}{c}{bias}&\multicolumn{1}{c}{MSE}&\multicolumn{1}{c}{bias}&\multicolumn{1}{c}{MSE}&\multicolumn{1}{c}{bias}&\multicolumn{1}{c}{MSE}\\
  \hline
  MLE&$ 0.03610$&$0.07872$&$ 0.13328$&$0.91014$&$ 0.10132$&$0.44693$\\
  AUE&$-0.00058$&$0.07568$&$-0.02003$&$0.88393$&$-0.00263$&$0.43295$\\
  \hline\\
  
  &\multicolumn{2}{c}{$(\mu,\sigma)=(5,0.1)$}&\multicolumn{2}{c}{$(\mu,\sigma)=(0,5)$}&\multicolumn{2}{c}{$(\mu,\sigma)=(5,5)$}\\
  &\multicolumn{2}{c}{$t^2=61.85059$}&\multicolumn{2}{c}{$t^2=5.18058$}&\multicolumn{2}{c}{$t^2=10.69656$}\\
  &\multicolumn{1}{c}{bias}&\multicolumn{1}{c}{MSE}&\multicolumn{1}{c}{bias}&\multicolumn{1}{c}{MSE}&\multicolumn{1}{c}{bias}&\multicolumn{1}{c}{MSE}\\
  \hline
  MLE&$ 0.22886$&$2.57769$&$-0.01906$&$0.20499$&$ 0.03608$&$0.43819$\\
  AUE&$-0.05227$&$2.51639$&$ 0.00126$&$0.20603$&$-0.00145$&$0.42903$\\
  \hline
\end{tabular}
\end{table}

\section{Discussion}
\label{sect:disc}
We have discussed the asymptotic bias reduction of maximum
likelihood estimates of generic estimands with parameter
estimates obtained by maximizing suitable penalized likelihoods.
The penalty slightly denormalizes the model, but
no difficulty arises for explicit calculations.
We have answered the two questions raised in Introduction.

The first question concerns bias reduction of generic estimands.
We demonstrated that the problem of finding a penalty function
that gives an estimator with an asymptotic bias of $o(n^{-1})$
can be boiled down to the integration of a quasi-linear partial
differential equation \eqref{pde}.
Since the partial differential equation has a solution, we can
obtain the desired penalty function for generic model manifolds
and estimands. 

The second question concerns the system of partial differential
equations \eqref{firth_pen} obtained by Firth \cite{Fir93}.
For estimates of parameters, we show that our bias-reduction
method reproduces Firth's bias-reduction method.
We pointed out that the system is overdetermined, except for
one-dimensional models. The integration of an overdetermined
system requires an integrability condition. Flat model manifolds
are exceptions; we could assess the integrability of
the overdetermined system, and we obtained some explicit results.

Another natural question is asking what estimand is
asymptotically unbiased for a given penalty, because a penalty
function can be given a priori, for example, as a regularizer
or a prior. This question would be studied in the context of
integral geometry, because the partial differential equation
for an estimand function is a variant of the Laplace equation
and the integration relies on the group isometries of manifolds
(see Chapter II of \cite{Hel84}). We left this issue in future
research. 

Finally, we comment on the concept of $\alpha$-Laplacians
introduced in \eqref{al_Lap}. We introduced this concept
because it simplifies various arguments in this study.
This simplification is not restricted to bias reduction.
For a curved exponential family, Komaki \cite{Kom96} constructed
an optimal predictive distribution in terms of the Kullback-Leibler
loss by shifting the distribution in a direction orthogonal to
the model manifold with the amount
$\Delta^{(-1)}p(\cdot;\hat{\xi}_{\rm MLE})/2$
(the last equation on page 307 of \cite{Kom96}). This expression
implies that if the distribution function is $(-1)$-harmonic,
optimality has already been achieved. Further investigation of
$\alpha$-Laplacians might be interesting.

\begin{appendix}
  
  \section{Proofs}

  In this appendix, we collect proofs and concepts in differential
  geometry used in the proofs. For general background information
  on differential geometry, see \cite{KN63}. A concise summary of
  the terminology in affine differential geometry is in Chapter I
  of \cite{NS94}.
  
  \subsection{Sections~2.1-2.2}

  \begin{proof}[Proof of Proposition 2.1]  
    By the rule of covariant differentiation of tensor fields
    (see Section III.7 of \cite{KN63}), covariant derivative of
    the metric tensor is
    \[
    \nabla^{(\alpha)}_i g_{jk}
    =\partial_i g_{jk}
    -\sum_r{\Gamma^{(\alpha)}}^r_{ij}g_{rk}
    -\sum_t{\Gamma^{(\alpha)}}^t_{ik}g_{jt}.
    \]
    This expression is equal to
    \begin{align}    
    &\partial_i g_{jk}
    -\sum_{r,s}g^{rs}\Gamma^{(\alpha)}_{ij,s}g_{rk}
    -\sum_{t,u}g^{tu}\Gamma^{(\alpha)}_{ik,u}g_{jt}
    =\partial_i g_{jk}
    -\sum_s\delta^s_k\Gamma^{(\alpha)}_{ij,s}
    -\sum_u\delta^t_j\Gamma^{(\alpha)}_{ik,t}\nonumber\\
    &=\partial_i g_{jk}
    -\Gamma^{(\alpha)}_{ij,k}
    -\Gamma^{(\alpha)}_{ik,j}.\label{prop2.1-1}
    \end{align}
    Differentiation of $\mathbb{E}_\xi[u_k]=0$ leads to
    $g_{jk}=-\mathbb{E}_\xi[\partial_ju_k]=\mathbb{E}_\xi[u_ju_k]$,
    since
    \begin{align*}
      0&=\partial_j\int_{\mathcal{X}^n}u_k(\xi;x)e^{l(\xi;x)}dx
      =\int_{\mathcal{X}^n}
        \{\partial_ju_k(\xi;x)e^{l(\xi;x)}+
        u_j(\xi;x)u_k(\xi;x)e^{l(\xi;x)}\}dx\\
     &=\mathbb{E}_\xi[\partial_ju_k]+\mathbb{E}_\xi[u_ju_k].
    \end{align*}
    In the same manner, $g_{jk}=\mathbb{E}_\xi[u_ju_k]$ gives
    \[
    \partial_i g_{jk}
    =\mathbb{E}_\xi[(\partial_i\partial_j l)u_k]
    +\mathbb{E}_\xi[(\partial_i\partial_k l)u_j]
    +\mathbb{E}_\xi[u_iu_ju_k].
    \]
    Substituting this expression into \eqref{prop2.1-1} and
    using the definition of the $\alpha$-connection (5), we have
    $\nabla^{(\alpha)}_i g_{jk}=\alpha S_{ijk}$.
  \end{proof}

  See Appendix~\ref{appe:Lem2.2} for the proof of Lemma~2.2.

  \color{black}
  
\begin{proof}[Proof of Lemma~2.4]
By Lemma~2.2, the bias is evaluated as
\begin{align*}
  \mathbb{E}_\xi[f(\hat{\xi})-f(\xi)]=&
  \sum_i\mathbb{E}_\xi[(\hat{\xi}-\xi)^i]\partial_i f(\xi)+
  \sum_{i,j}\mathbb{E}_\xi[(\hat{\xi}-\xi)^i(\hat{\xi}-\xi)^j]
  \frac{1}{2}\partial_i\partial_jf(\xi)\\
  &+\sum_{i,j,k}\mathbb{E}_\xi[(\hat{\xi}-\xi)^i(\hat{\xi}-\xi)^j(\hat{\xi}-\xi)^k]
  \frac{1}{3!}\partial_i\partial_j\partial_kf(\tilde{\xi})\\
  =&\sum_{i,j}g^{ij}\left\{\left(\partial_j \tilde{l}
    -\frac{1}{2}\sum_{k,r}g^{kr}\Gamma^{(-1)}_{kr,j}\right)
    \partial_if+\frac{1}{2}\partial_i\partial_jf\right\}
  +o(n^{-1}),
\end{align*}
where $\tilde{\xi}$ is a point between $\hat{\xi}$ and $\xi$.
The assertion follows from the definitions of the $(-1)$-Laplacian
(9). The evaluation of the last term of the middle
expression is similar to that of $R_{1,i}$ in the proof of
Lemma~2.2.
\end{proof}
  
Recall the Lehmann-Scheff\'e theorem:

\begin{theorem}[Theorem 2.1.11, \cite{LC98}]\label{theo:UMVUE1}
  Let a sample be distributed according to a parametric model
  $M=\{p(\cdot;\xi):\xi\in\Xi\}$, and suppose that $t$ is
  a complete sufficient statistic for $M$.
  \begin{itemize}
  \item[$(i)$] For every $U$-estimable estimand $f(\xi)$,
    $\xi\in\Xi$, there exists a UMVUE.
  \item[$(ii)$] The UMVUE in $(i)$ is the unique unbiased estimator
    that is a function of $t$.
  \end{itemize}
\end{theorem}

\begin{proof}[Proof of Theorem~2.6]
  By the factorization criterion (Theorem 1.6.5 of \cite{LC98}),
  if we have a statistic $t$ to be sufficient for a family of
  probability densities $M$ of a sample $x$, there exist
  non-negative functions $h_\xi$ and $k$ such that the density
  of $p(\cdot;\xi)$ satisfy $p(x;\xi)=h_\xi(t(x))k(x)$. Since
  $k$ does not depend on $\xi$, the maximum likelihood estimator
  is a function of $t$. On the other hand, since the penalty
  $\tilde{l}$ does not depend on $x$, the maximizer of
  the penalized likelihood $e^{\tilde{l}(\xi)}h_\xi(t(x))$
  modulo constant, is also a function of $t$. Let the maximizer
  be denoted by $\hat{\xi}(t)$. According to
  Theorem~\ref{theo:UMVUE1}, the UMVUE of $f(\xi)$ exists
  uniquely and is a function of $t$. Let the UMVUE be denoted
  by $\delta(t)$. Now,
  $\mathbb{E}_\xi[f(\hat{\xi}(t))-\delta(t)]=\mathbb{E}_\xi f(\hat{\xi}(t))-f(\xi)=o(n^{-1})$,
  otherwise $f(\hat{\xi}(t))$ should have bias of $O(n^{-1})$,
  which contradicts to Lemma~2.4.
\end{proof}  

\subsection{Proof of Lemma~2.2}\label{appe:Lem2.2}

Consider a log-likelihood function
$l(\xi;x_1,...,x_n):=\sum_{i=1}^n \log p(x_i;\xi)$
of a sample $(x_1,...,x_n)\in\mathcal{X}^n$ with a parametric model
$p(x;\xi)$, $\xi\in\Xi$. The parameter space $\Xi$ is an open
subset of $\mathbb{R}^d$ for a fixed $d\in\mathbb{Z}_{>0}$ for large $n$.
The true value of the parameter, $\xi_0$, is assumed to be in the
interior of $\Xi$. Expectations are taken with respect to the product
probability measure $P_{\xi_0}(dx)=e^{l(\xi_0;x)}\prod_{i=1}^ndx_i$
and a derivative is denoted by $\partial_i:=\partial/\partial\xi_i$.

We prepare regularity conditions:

\begin{itemize}

\item[A1:] The map $\mathcal{X}\ni x\mapsto l(\xi;x)$ is measurable for
  each $\xi\in\Xi$;

\item[A2:] The map $\Xi\ni \xi\mapsto l(\xi;x)$ is three times differentiable
  for each $\xi\in\Xi$;

\item[A3:] $\partial_il(\xi_0;x)$, $i\in\{1,\ldots,d\}$
  are square integrable with respect to $P_{\xi_0}(dx)$;
  
\item[A4:] Let $G=(g_{ij}(\xi_0))$,
  $g_{ij}(\xi_0):=\mathbb{E}[\sum_{k=1}^n\partial_i \log p(x_k;\xi_0)\partial_j \log p(x_k;\xi_0)]$.
  The largest eigenvalue of $-C^{-1}GC^{-1}$, $\lambda_{\text{max}}$,
  for a matrix
  $C=(c_{ij})=\text{diag}(c_1,...,c_d)$ with $c_i>0$ such that
  $c_*:=\min_i c_i\to\infty$ as $n\to\infty$ satisfies
  ${\lim\sup}_{n\to\infty}\lambda_{\text{max}}\in(-\infty,0)$;
  
\item[A5:] For an $r\in\mathbb{Z}_{>0}$, the $r$-th moments of
  the following are bounded:
  \[
  \frac{1}{c_i}|\partial_i l(\xi_0;x)|, \quad
  \frac{1}{\sqrt{c_ic_j}}|\partial_i\partial_j l(\xi_0;x)+g_{ij}(\xi_0)|, \quad
  \frac{c_*}{c_ic_jc_k}M_{ijk}(\xi_0),
  \]
  where
  $M_{ijk}(\xi_0):=\sup_{\tilde{\xi}\in B_\delta(\xi_0)}|\partial_i\partial_j\partial_k l(\tilde{\xi};x)|$
  with a ball
  \[
  B_\delta(\xi_0):=\{\tilde{\xi}:|\tilde{\xi}-\xi_0|\le \delta c_*/c_i,i\in\{1,...,d\}\}.\]
  
\end{itemize}

Under conditions A1-A5, Das et al. \cite{DJR04} proved, as their Theorem 2.1,
the following theorem for an asymptotic representation of
$\hat{\xi}-\xi_0$ to study the mean squared error of the empirical
predictor of general linear mixed-effects models, where
$\hat{\xi}$ is the solution of the system of score equations
$u_i(\xi;x)=0$, $i\in\{1,\ldots,d\}$.

\begin{theorem}[\cite{DJR04}]\label{theo:DJR}
  Under the regularity conditions A1-A5,
  \begin{itemize}    
  \item[$(i)$] A $\hat{\xi}\in\Xi$ exists such that for any $\rho\in(0,1)$
    there is a set of events $\mathcal{E}$ satisfying for large $n$ and on
    $\mathcal{E}$,
    $\partial_il(\hat{\xi};x)=0$, $|c_{ij}(\hat{\xi}-\xi_0)^j|<c_*^{1-\rho}$,
    and
    \[
      \hat{\xi}^i=\xi_0^i+\sum_jg^{ij}(\xi_0)\partial_jl(\xi_0;x)+R,
      \quad |R|\le c_*^{-2\rho}u_*, \quad i\in\{1,\ldots,d\} 
    \]
    with $\mathbb{E}u_*^r$ bounded;
  \item[$(ii)$] $\mathbb{P}(\mathcal{E}^c)\le c_0c_*^{-\tau r}$ for some constant $c_0$
    and $\tau:=(1/4)\wedge (1-\rho)$.  
  \end{itemize}
\end{theorem}

This theorem states that the solution of the system of score
equations $u_i(\xi;x)=0$ exists, and lies in the parameter space $\Xi$
with probability tending to one.

Consider the penalized log-likelihood of a sample
$(x_1,...,x_n)\in\mathcal{X}^n$ with a penalty function
$\tilde{l}(\xi)=O(1)$:
\[
  l^*(\xi;x_1,...,x_n)
  :=l(\xi;x_1,...,x_n)+\tilde{l}(\xi), \quad \xi\in\Xi.
\]
The maximizer of $l^*(\xi;x_1,...,x_n)$ is denoted by $\hat{\xi}_n$.
Expectations are still taken with respect to the product probability
measure, $P_{\xi_0}(dx)=e^{l(\xi_0;x)}\prod_{i=1}^ndx_i$, and we have
\[
\mathbb{E}[\sum_{k=1}^n\partial_i l^*(\xi_0;x_k)\partial_jl^*(\xi_0;x_k)]=g_{ij}(\xi_0)+O(1).
\]
Theorem~\ref{theo:DJR} is used in the following proof of
Lemma~2.2.

We prepare the regularity conditions with some modifications.
Das et al. \cite{DJR04} needed conditions A4 and A5 to treat multiple
asymptotics.
In contrast, a single asymptotic is sufficient for our purpose.
Therefore, we set $c_1=\cdots =c_n=c_*=\sqrt{n}$, and the resulting
conditions are conditions B4 and B5.

\begin{itemize}

\item[B1:] The map $\mathcal{X}\ni x\mapsto l^*(\xi;x)$ is measurable for
  each $\xi\in\Xi$;

\item[B2:] The map $\Xi\ni \xi\mapsto l^*(\xi;x)$ is four times differentiable
  for each $\xi\in\Xi$;

\item[B3:] $\partial_il(\xi_0;x)$, $i\in\{1,\ldots,d\}$ are square integrable
  with respect to $P_{\xi_0}(dx)$;

\item[B4:] The smallest eigenvalue of $G/n$, $\lambda_{\rm min}$,
  satisfies $\lim\inf_{n\to\infty}\lambda_{\rm min}\in(0,\infty)$.
  
\item[B5:] For an $r\in\mathbb{Z}_{\ge 9}$, the $r$-th moments of the following
  are bounded:
  \begin{align*}
  &\frac{1}{\sqrt{n}}|\partial_i l^*(\xi_0;x)|, \quad
   \frac{1}{n}|\partial_i\partial_j\partial_k\partial_r l^*(\xi_0;x)|, \quad
   \frac{1}{\sqrt{n}}|\partial_i\partial_j l^*(\xi_0;x)+g_{ij}(\xi_0)|, \\
   &\frac{1}{n}|\partial_i\partial_j\partial_k l^*(\xi_0;x)|\quad
    \frac{1}{n}|\partial_i\partial_j\partial_k l^*(\xi_0;x)
    -\mathbb{E}[\partial_i\partial_j\partial_k l^*(\xi_0;x)]|.
  \end{align*}

\item[B6:]
  $\max_{i\in\{1,...,d\}}|\hat{\xi}_n^i|<d_0n^s$ for some constant
  $d_0$ and $0<s<r/16-1/2$.
  
\end{itemize}

\begin{proof}[Proof of Lemma~2.2]
  With setting $\rho\in(2/3,3/4)$, Theorem~\ref{theo:DJR} concludes that
  \begin{itemize}
  \item[(a)] A $\hat{\xi}_n\in\Xi$ exists such that there is a set of events
  $\mathcal{E}$ satisfying for large $n$ and on
  $\mathcal{E}$, $\partial_il^*(\hat{\xi}_n;x)=0$,
  $|(\hat{\xi}_n-\xi_0)^j|<n^{-\rho/2}$ and
  \begin{equation}
    \hat{\xi}_n^i=\xi_0^i+\sum_jg^{ij}(\xi_0)\partial_jl^*(\xi_0;x)+R^i,
    \quad |R^i|\le n^{-\rho}u_*, \quad i\in\{1,\ldots,d\} \label{B1}
  \end{equation}
  with $\mathbb{E}(u_*^r)$ bounded;
  \item[(b)]
    $\mathbb{P}(\mathcal{E}^c)\le c_0n^{-r/8}$ for some constant $c_0$.
   \end{itemize} 
  For simplicity of expressions, in the following expressions,
  $g_{ij}(\xi_0)$ will be denoted by $g_{ij}$.
  In addition, $\mathbb{E}^\mathcal{E}(\cdot)$ and
  $\mathbb{E}^{\mathcal{E}^c}(\cdot)$ will denote
  $\mathbb{E}(\cdot1_{\mathcal{E}})$ and 
  $\mathbb{E}(\cdot1_{\mathcal{E}^c})$, respectively. By Taylor's theorem,
  \begin{align*}
  \partial_i l^*(\hat{\xi_n};x)-\partial_i l^*(\xi_0;x)
  =&\sum_j(\hat{\xi_n}-\xi_0)^j\partial_i\partial_j l^*(\xi_0;x)\nonumber\\
  &+\frac{1}{2}\sum_{j,k}(\hat{\xi_n}-\xi_0)^j(\hat{\xi_n}-\xi_0)^k
  \partial_i\partial_j\partial_kl^*(\xi_0;x)+R_{1,i}\nonumber\\
  =&-\sum_j(\hat{\xi_n}-\xi_0)^j g_{ij}+\sum_j(\hat{\xi_n}-\xi_0)^j
  \{\partial_i\partial_jl^*(\xi_0;x)+g_{ij}\}\nonumber\\
  &+\frac{1}{2}\sum_{j,k}(\hat{\xi_n}-\xi_0)^j(\hat{\xi_n}-\xi_0)^k
  \partial_i\partial_j\partial_kl^*(\xi_0;x)+R_{1,i},
  \end{align*}
  where
  \[
  R_{1,i}=\frac{1}{3!}\sum_{j,k,r}
  (\hat{\xi_n}-\xi_0)^j(\hat{\xi_n}-\xi_0)^k(\hat{\xi_n}-\xi_0)^r
  \partial_i\partial_j\partial_k\partial_rl^*(\tilde{\xi};x) \quad
  \]
  for a point $\tilde{\xi}$ between $\hat{\xi_n}$ and $\xi_0$.
  Since $\partial_jl^*(\hat{\xi}_n;x)=0$, we have
  \begin{align}
  (\hat{\xi_n}-\xi_0&)^i=\sum_jg^{ij}\left\{
  \partial_j l^*(\xi_0;x)+
  \sum_k(\hat{\xi_n}-\xi_0)^k(\partial_j\partial_kl^*(\xi_0;x)+g_{jk})
  \right.\nonumber\\
  &\left.+\frac{1}{2}\sum_{k,r}(\hat{\xi_n}-\xi_0)^k(\hat{\xi_n}-\xi_0)^r
  \partial_j\partial_k\partial_rl^*(\xi_0;x)+R_1\right\}
  =\sum_jg^{ij}\partial_jl^*(\xi_0;x)+R_2^i. \label{B3}
  \end{align}
  Here, $g^{ij}=O(n^{-1})$ holds by condition B4 and \eqref{B1}
  gives $|R_2^i|\le n^{-\rho}u_*$ with $\mathbb{E}^\mathcal{E}(u_*)$ bounded.
  Then, replacing every occurrence of $(\hat{\xi}_n-\xi_0)^i$ in the middle
  expression of \eqref{B3} by $\sum_jg^{ij}\partial_jl^*(\xi_0;x)+R_2^i$
  and taking the expectation of the first equality on $\mathcal{E}$,
  we have
  \begin{align*}
  \mathbb{E}^\mathcal{E}(\hat{\xi_n}-\xi_0)^i
      =&\sum_jg^{ij}\mathbb{E}^\mathcal{E}[\partial_jl^*(\xi_0;x)]
      +\sum_jg^{ij}\mathbb{E}^\mathcal{E}(R_{1,j})\\
  &+\sum_{j,k}g^{ij}
    \mathbb{E}^\mathcal{E}[(\sum_rg^{kr}\partial_rl^*(\xi_0;x)+R_2^k)
    (\partial_j\partial_kl^*(\xi_0;x)+g_{jk})]\\
  &+\sum_{j,k,r,s,t} g^{ij}
    \frac{1}{2}\mathbb{E}^\mathcal{E}
    [(g^{ks}\partial_sl^*(\xi_0;x)+R_2^k)
      (g^{rt}\partial_tl^*(\xi_0;x)+R_2^r)
      \partial_j\partial_k\partial_rl^*(\xi_0;x)].
  \end{align*}
  The Cauchy-Schwarz inequality, condition B5, and $|R_2^k|\le n^{-\rho}u_*$
  lead to
  \begin{align*}
    \sum_{j,k}g^{ij}
    |\mathbb{E}^\mathcal{E}[R_2^k(\partial_j\partial_kl^*(\xi_0;x)+g_{jk})]|
  &\le\sum_{j,k}g^{ij}(\mathbb{E}^\mathcal{E}(R_2^k)^2)^{1/2} 
  \{\mathbb{E}^\mathcal{E}(\partial_j\partial_kl^*(\xi_0;x)+g_{jk})^2\}^{1/2}\\
  &\le\{\mathbb{E}^\mathcal{E}(u_*^2)\}^{1/2}o(n^{-1}),
  \end{align*}
  In the same way, we observe that
  \[
  \sum_{j,k,r,s}g^{ij}g^{ks}
  \mathbb{E}^\mathcal{E}[\partial_sl^*(\xi_0;x)R_2^r
  \partial_j\partial_k\partial_rl^*(\xi_0;x)],\quad
  \sum_{j,k,r}g^{ij}\mathbb{E}^\mathcal{E}[R_2^rR_2^k\partial_j\partial_k\partial_rl^*(\xi_0;x)]
  \]
  are $o(n^{-1})$. Condition B5 and $|(\hat{\xi}-\xi_0)^i|<n^{-\rho/2}$
  lead to
  \[
  \sum_jg^{ij}|\mathbb{E}^\mathcal{E}(R_{1,j})|<
  \sum_{j,k,r,s}g^{ij}\frac{1}{3!}n^{-3\rho/2}
  |\mathbb{E}^\mathcal{E}[\partial_j\partial_k\partial_r\partial_sl^*(\tilde{\xi};x)]|=o(n^{-1}),
  \]
  where $\rho>2/3$ is demanded. Therefore, we have
  \begin{align*}
  \mathbb{E}^\mathcal{E}(\hat{\xi_n}&-\xi_0)^i
  =2\sum_jg^{ij}\mathbb{E}^\mathcal{E}[\partial_jl^*(\xi_0;x)]
   +\sum_{j,k,r}g^{ij}g^{kr}\mathbb{E}^\mathcal{E}[\partial_rl^*(\xi_0;x)
    \partial_j\partial_kl^*(\xi_0;x)]\\
   &+\sum_{j,k,r,s,t}g^{ij}g^{ks}g^{rt}\frac{1}{2}
   \mathbb{E}^\mathcal{E}[\partial_sl^*(\xi_0;x)\partial_tl^*(\xi_0;x)
   \partial_j\partial_k\partial_rl^*(\xi_0;x)]+o(n^{-1})\\
   =&2\sum_jg^{ij}\mathbb{E}^\mathcal{E}[\partial_jl^*(\xi_0;x)]
   +\sum_{j,k,r}g^{ij}g^{kr}\mathbb{E}^\mathcal{E}[\partial_rl^*(\xi_0;x)
    \partial_j\partial_kl^*(\xi_0;x)]\\
   &+\sum_{j,k,r}g^{ij}g^{kr}\frac{1}{2}
    \mathbb{E}^\mathcal{E}[\partial_j\partial_k\partial_rl^*(\xi_0;x)]+o(n^{-1})\\
   =&2\sum_jg^{ij}\mathbb{E}^\mathcal{E}[\partial_jl^*(\xi_0;x)]
   +\sum_{j,k,r}g^{ij}g^{kr}\mathbb{E}^\mathcal{E}[\{\partial_r l(\xi_0;x)+\partial_r \tilde{l}(\xi_0)\}
    \{\partial_j\partial_k(l(\xi_0;x)+\tilde{l}(\xi_0))\}]\\
   &+\sum_{j,k,r}g^{ij}g^{kr}\frac{1}{2}
   \mathbb{E}^\mathcal{E}[\partial_j\partial_k\partial_rl^*(\xi_0;x)]+o(n^{-1}),
   \end{align*}
  where the second equality holds from the fact which follows from
  condition B5:
   \[
   cov^\mathcal{E}(\partial_sl^*(\xi_0;x)\partial_tl^*(\xi_0;x),
   \partial_j\partial_k\partial_rl^*(\xi_0;x))=o(n^2).
   \]  
   Since $\mathbb{E}^\mathcal{E}(\cdot)\le \mathbb{E}(\cdot)$, the last
   expression is bounded from the above by
   \begin{align}
   &2\sum_jg^{ij}\mathbb{E}[\partial_jl^*(\xi_0;x)]
   +\sum_{j,k,r}g^{ij}g^{kr}\mathbb{E}[\{\partial_r l(\xi_0;x)+\partial_r \tilde{l}(\xi_0)\}
    \{\partial_j\partial_k(l(\xi_0;x)+\tilde{l}(\xi_0))\}]\nonumber\\
   &+\sum_{j,k,r}g^{ij}g^{kr}\frac{1}{2}
   \mathbb{E}[\partial_j\partial_k\partial_rl^*(\xi_0;x)]+o(n^{-1}). \label{B4}
   \end{align}
   On the other hand, by (b) and condition B6, we observe
   \[
   \mathbb{E}^{\mathcal{E}^c}(\hat{\xi}_n-\xi_0)^i
   \le d_0 n^s \mathbb{P}(\mathcal{E}^c)\le c_0d_0 n^{s-r/8}=o(n^{-1}).
   \]
   and thus
   $\mathbb{E}(\hat{\xi}_n-\xi_0)^i=\mathbb{E}^{\mathcal{E}}(\hat{\xi}_n-\xi_0)^i+o(n^{-1})$.
   Therefore, \eqref{B4} gives
   \begin{align*}
     &\mathbb{E}(\hat{\xi}_n-\xi_0)^i=
     2\sum_jg^{ij}\partial_j\tilde{l}(\xi_0)\\
     &+\sum_{j,k,r}g^{ij}g^{kr}  
     \left\{
     \mathbb{E}[\partial_rl(\xi_0;x)\partial_j\partial_kl(\xi_0;x)]+
     \partial_r\tilde{l}(\xi_0)\mathbb{E}[\partial_j\partial_k l(\xi_0;x)]
     +\frac{1}{2}\mathbb{E}[\partial_j\partial_k\partial_rl(\xi_0;x)]\right\}\\
    &+o(n^{-1})\\
   &=
   \sum_jg^{ij}\partial_j \tilde{l}(\xi_0)
   +\sum_{j,k,r}g^{ij}g^{kr}\left\{\mathbb{E}[\partial_r l(\xi_0;x)
   \partial_j\partial_kl(\xi_0;x)]
   +\frac{1}{2}
   \mathbb{E}[\partial_j\partial_k\partial_rl(\xi_0;x)]\right\}
   +o(n^{-1})\\
   &=\sum_jg^{ij}\left\{\partial_j \tilde{l}(\xi_0)
   +{\frac{1}{2}}\sum_{k,r}g^{kr}(\Gamma^{(1)}_{jk,r}-\partial_rg_{jk})\right\}
   +o(n^{-1})
   \end{align*}
   for large $n$, where the second equality follows by
   $\mathbb{E}[\partial_j\partial_kl(\xi_0;x)]=-g_{jk}$.
   The last equality follows by
   $\mathbb{E}[\partial_j\partial_k\partial_rl(\xi_0;x)]=-\partial_rg_{jk}-\Gamma^{(1)}_{jk,r}$, which can be obtained by differentiating
   $g_{jk}=-\mathbb{E}[\partial_j\partial_kl(\xi_0;x)]$ and the
   definition of the $\alpha$-connection (5).
   \color{black}
   Hence, assertion (i) is established.
   For assertion (ii), by using \eqref{B1}, we have
   \[
    \mathbb{E}[(\hat{\xi_n}-\xi_0)^i(\hat{\xi_n}-\xi_0)^j]=
    \sum_{k,r}g^{ik}g^{jr}\mathbb{E}[\partial_kl^*(\xi;x)\partial_rl^*(\xi;x)]+o(n^{-1})=g^{ij}+o(n^{-1})
  \]
  in a similar way as in the proof of assertion (i).
\end{proof}

\subsection{Section~2.4: Geodesic distances and normal coordinates}

In a local coordinate system $\{\xi^1,\cdots,\xi^d\}$,
consider a curve $\gamma=\xi(t)$, $a<t<b$, where
$-\infty\le a<b\le \infty$, of class $C^1$ in a manifold $M$.
The {\it length} of $\gamma$ is defined as
$\int_a^b\sqrt{\sum_{i,j}g_{ij}\dot{\xi}^i\dot{\xi}^j}dt$,
where $\dot{\xi}(t)=d\xi/dt$ denotes the vector
  tangent to $\gamma$ at $\xi(t)$.
A curve $\gamma$ is called a {\it geodesic} if the vector field
$X=\dot{\xi}(t)$ defined along $\gamma$ is parallel along $\gamma$,
that is, if $\nabla^{(0)}_XX$ exists and equals 0 for all $t$.
In a local coordinate system, the geodesic equation
is expressed as
\begin{equation}\label{geodesic}
  \ddot{\xi}^i+\sum_{j,k}\Gamma^i_{jk}\dot{\xi}^j\dot{\xi}^k=0,
  \quad i\in\{1,\ldots,d\},
\end{equation}
where $\ddot{\xi}=d^2\xi/dt^2$. The parameter $t$ is normalized
such that we have
\begin{equation}\label{can}
\sum_{i,j}g_{ij}\dot{\xi}^i\dot{\xi}^j=1
\end{equation}
and called the {\it canonical parameter} of the geodesic $\gamma$.
The canonical parameter of a geodesic should not be confused
with that of an exponential family. The {\it distance}
${\rm dis}(\zeta,\xi)$ on $M$ is the infimum length of all
piecewise differentiable curves of class $C^1$ joining $\zeta$ and
$\xi$ in $M$. Let the neighborhood of $\zeta$ in $M$ be $U_\zeta$.
With the 0-connection, it can be shown that every point
$\xi\in U_\zeta$ can be joined to $\zeta$ by the unique geodesic
lying in $U_\zeta$, and the length is equal to ${\rm dis}(\zeta,\xi)$
(Proposition IV.3.4 of \cite{KN63}). In this sense,
we call this distance the {\it geodesic distance}. In an Euclidean
manifold, the geodesic is the straight line joining $\zeta$ and
$\xi$.

An orthonormal frame at $\zeta\in M$ defines a coordinate
system in the tangent space $T_\zeta M$. The diffeomorphism
of the neighborhood of $\zeta$ in $T_\zeta M$ to the neighborhood
of $\zeta$ in $M$, $U_\zeta$, defines a local coordinate system
in $U_\zeta$. We denote the local coordinate system by
$\{\zeta^1,\ldots,\zeta^d\}$ and call it
a {\it normal coordinate system} at $\zeta$. Note that
$\{\partial/\partial \zeta^1,\ldots,\partial/\partial \zeta^d\}$
forms an orthonormal frame at $\zeta$, but may not be
orthonormal at other points. If $\{\zeta^1,\ldots,\zeta^d\}$ is
a normal coordinate system at $\zeta$, then the geodesic
$\gamma=\zeta(t)$ with the initial condition $\zeta(0)=0$ and
$\dot{\zeta}(0)=c$ is expressed as $\zeta^i=c^it$,
$i\in\{1,\ldots,d\}$ (see Proposition III.8.3 of \cite{KN63}).

Formulas (i) and (ii) in the following lemma appear as equations
(55) and (57), respectively, in Chapitre VII of \cite{Rie49}.
Proofs are given for readers' convenience, because the proofs were
not explicitly given in \cite{Rie49}.

For a geodesic $\gamma=\xi(t)$,
the length minimizing property of geodesics gives the following
useful relation
\begin{equation}\label{riesz52}
  \frac{\partial t}{\partial\xi^i}=\sum_jg_{ij}\dot{\xi}^j,
\end{equation}
which is derived as equation (52) in Chapitre VII of \cite{Rie49}.

\begin{lemma}\label{lemm:riesz}
  For a function $\varphi(t^2)$, where $t={\rm dis}(\zeta,\xi)$ is
  the geodesic distance, we have
  \begin{itemize}

  \item[(i)] $\Delta^{(0)}\varphi(t^2)=\varphi'(t^2)\Delta^{(0)}(t^2)+4t^2\varphi''(t^2)$, and

  \item[(ii)]
    $\displaystyle\Delta^{(0)}(t^2)=2d+t\frac{d}{dt}\{\log h(\zeta)\}$,
    where $\zeta=\dot{\xi}(0)t$ is the normal coordinate system
    at $\zeta$, and $h(\zeta)$ is the determinant of
    the Fisher metric tensor.
  \end{itemize}
  
\end{lemma}

\begin{proof}
  For assertion (i), by the definition of the Laplace--Beltrami operator (11), we have
\begin{align*}
  \Delta^{(0)}\varphi(t^2)&=
  \frac{1}{\sqrt{g}}\sum_i\frac{\partial}{\partial\xi^i}
  \left(\sqrt{g}\sum_jg^{ij}\frac{\partial \varphi(t^2)}{\partial\xi^j}\right)
  =\frac{1}{\sqrt{g}}\sum_i\frac{\partial}{\partial \xi^i}
  \left(\sqrt{g}\sum_jg^{ij}\frac{\partial t^2}{\partial\xi^j}\varphi'(t^2)\right)\\
  &=\frac{1}{\sqrt{g}}\sum_i\frac{\partial}{\partial\xi^i}
  \left(\sqrt{g}\sum_jg^{ij}\frac{\partial t^2}{\partial\xi^j}\right)
  \varphi'(t^2)
  +\sum_{i,j}g^{ij}\frac{\partial t^2}{\partial\xi^i}\frac{\partial t^2}
  {\partial\xi^j}\varphi''(t^2)\\
  &=\varphi'(t^2)\Delta^{(0)}(t^2)+4\sum_{i,j}g_{ij}\dot{\xi}^i\dot{\xi}^j
  t^2\varphi''(t^2)=\varphi'(t^2)\Delta^{(0)}(t^2)+4t^2\varphi''(t^2),
\end{align*}
where the second last equality follows by \eqref{riesz52} and
  the last equality follows by \eqref{can}.
For assertion (ii), we work with the normal coordinate system
 at $\zeta$. We have
\begin{align*}
  \Delta^{(0)}(t^2)&=
  \frac{1}{\sqrt{h}}\sum_i\frac{\partial}{\partial \zeta^i}
  \left(\sqrt{h}\sum_jh^{ij}\frac{\partial t^2}{\partial\zeta^j}\right)
  =\frac{1}{\sqrt{h}}\sum_i\frac{\partial \sqrt{h}}{\partial\zeta^i}
  \sum_jh^{ij}\frac{\partial t^2}{\partial\zeta^j}+
  \sum_{i,j}\frac{\partial}{\partial \zeta^i}
  h^{ij}\frac{\partial t^2}{\partial\zeta^j}\\
  &=t\frac{d}{dt}\{\log h(\zeta)\}\sum_{i,j}h_{ij}\dot{\zeta}^i\dot{\zeta}^j
  +2\sum_i\frac{\partial}{\partial\zeta^i}(t\dot{\zeta}^i)
  =t\frac{d}{dt}\{\log h(\zeta)\}
  +2\sum_i\frac{\partial}{\partial\zeta^i}(t\dot{\zeta}^i),
\end{align*}
where the second last equality follows by \eqref{riesz52}.
Because \eqref{can} and
\[
\sum_i\frac{\partial}{\partial\zeta^i}(t\dot{\zeta}^i)
        =\sum_i\frac{\partial}{\partial\zeta^i}(\dot{\xi}^i(0)t)
  =\sum_i\frac{\partial\zeta^i}{\partial\zeta^i}=d,
\]
the assertion holds.
\end{proof}

\begin{proof}[Proof of Theorem~2.9]
  We work with the normal coordinate system at $\zeta$
  denoted by $\zeta=\dot{\xi}(0)t$. By the relation between
  $\alpha$-Laplacians (11), we have
  \[
    \Delta^{(-1)}f=\Delta^{(0)}f-\frac{1}{2}\sum_iS^i\partial_i f.
  \]
  In the normal coordinate system at $\zeta$, the partial differential
  equation (14) is written as
  \begin{equation}\label{theo:cond0:1}
    \sum_{i,j}h^{ij}\frac{\partial f}{\partial\zeta^i}\frac{\partial l^*}{\partial\zeta^j}+\frac{1}{2}\Delta^{(0)}f=0,
  \end{equation}
  where
  \[
    l^*(\zeta):=\tilde{l}(\zeta)-\frac{1}{4}\sum_i\int S_i\dot{\xi}^idt.
  \]
  If the penalty function $\tilde{l}(\zeta)$ is a function of
  the squared geodesic distance, we may write $l^*(\zeta)=\chi^*(t^2(\zeta))$.
  Then, \eqref{theo:cond0:1} becomes
  \begin{equation}\label{theo:cond0:2}
    4t^2\sum_{i,j}h^{ij}\frac{\partial t}{\partial\zeta^i}
    \frac{\partial t}{\partial\zeta^j}\varphi'(t^2){\chi^*}'(t^2)
    +\frac{1}{2}
    \{\varphi'(t^2)\Delta^{(0)}(t^2)+4t^2\varphi''(t^2)\}=0,
  \end{equation}
      where $\varphi'(x)=d\varphi/dx$,  
        $\varphi''(x)=d^2\varphi/dx^2$, and $\chi^{*'}(x)=d\chi^*/dx$ 
        for $x\in\mathbb{R}_{>0}$
        and we used
  Lemma~\ref{lemm:riesz} (i). By using \eqref{can} and \eqref{riesz52},
  \eqref{theo:cond0:2} is recast into the ordinary differential
  equation for $\varphi$:
  \begin{equation}\label{theo:cond0:3}
    {\chi^*}'(t^2)=-\frac{1}{2}
    \left\{\frac{\varphi''(t^2)}{\varphi'(t^2)}+\frac{\Delta^{(0)}(t^2)}{4t^2}\right\}.
  \end{equation}
  Noting that Lemma~\ref{lemm:riesz} (ii) yields
  \[
    \int\frac{\Delta^{(0)}(t^2)}{t^2}dt^2=2d\log t^2+
    \int\frac{1}{t}\frac{d}{dt}\{\log h\}dt^2=
    2d\log t^2+2\log h+{\rm const.},
  \]
  we obtain
  $\chi^*(t^2)=-\log\left\{|\varphi'(t^2)|t^d\sqrt{h}\right\}/2+{\rm const.}$
  by integrating \eqref{theo:cond0:3}. Hence, (21) follows.
\end{proof}

\subsection{Section~3: Flat manifolds}

The Riemann curvature tensor field $R$ is defined through
\[
R(X,Y)Z:=\nabla_X\nabla_YZ-\nabla_Y\nabla_XZ-\nabla_{[X,Y]}Z,
\]
where $\nabla_X=\sum_ix^i\nabla_i$ and
$[X,Y]:=\sum_{i,j}(x^i\partial_i y^j-y^i\partial_i x^j)\partial_j$
for vector fields $X=\sum_ix^i\partial_i$ and $Y=\sum_iy^i\partial_i$.  
The components are introduced by
$R(\partial_j,\partial_k)\partial_i=:\sum_rR^r_{ijk}\partial_r$,
where
\begin{equation}\label{curvature}
  R^r_{ijk}=\partial_j\Gamma^r_{ki}-\partial_k\Gamma^r_{ji}
  +\sum_s\Gamma^s_{ki}\Gamma^r_{js}-\sum_t\Gamma^t_{ji}\Gamma^r_{kt}
\end{equation}
(see Proposition III.7.6 of \cite{KN63}). The Riemann curvature
tensor with respect to the $\alpha$-connection, that is,
the expression \eqref{curvature} with Christoffel symbols (6)
is specifically denoted by ${R^{(\alpha)}}^r_{ijk}$.

As stated before Definition 3.1 in the text, an affine
connection on $M$ is flat if and only if a local coordinate
system exists around each point such that $\Gamma^i_{jk}=0$
for all $i$, $j$, and $k$ \citep{NS94}. An equivalent condition
is that the curvature tensor field vanishes identically on $M$
(see Theorem II.9.1 of \cite{KN63}).

\begin{proposition}
  If a manifold $M$ is one-dimensional, any affine connection
  on $M$ is flat because the component \eqref{curvature} is
  anti-symmetric in indices $j$ and $k$; namely,
  $R^1_{111}=-R^1_{111}=0$. 
\end{proposition}


Consider an affine connection $\nabla$ on $M$. If there exists
a \textit{parallel volume element,} that is, a non-zero
$d$-dimensional differential form $v$ satisfying
$\nabla v=0$ around each point of the manifold $M$, then
the affine connection $\nabla$ is said to be {\it locally equiaffine}
(see Section I.3 of \cite{NS94}). In other words, if an affine
connection $\nabla$ on $M$ is locally equiaffine, the system of
partial differential equations $\nabla v=0$ for $v$ is integrable
everywhere.

For the relationship between the $\alpha$-flatness of a manifold
$M$ and the locally equiaffineness of the affine connection on $M$,
we have a technical lemma.
We do not use this lemma directly in this paper, but many fact
appear in this paper relate to the following elementary proof.

\begin{lemma}\label{lemm:TA}
  Consider a $C^\infty$-manifold $M$.
  \begin{itemize}
  \item[$(i)$] The $0$-connection on $M$ is locally equiaffine.
  \item[$(ii)$] If $M$ is $\alpha_0$-flat for
    $\alpha_0\in\mathbb{R}\setminus\{0\}$, then
    the $\alpha$-connection on $M$ is locally equiaffine for
    all $\alpha\in\mathbb{R}$.
  \item[$(iii)$] The converse of $(ii)$ is not true.
  \end{itemize}  
\end{lemma}

  \begin{proof}
  With a local coordinate system $\{\xi^1,\ldots,\xi^d\}$,
  a $d$-dimensional differential form $v$ is written as
  $v(\xi)=v_{i_1\ldots i_d}(\xi)
  d\xi^{i_1}\wedge\cdots\wedge d\xi^{i_d}$,
  where the density $v_{i_1\ldots i_d}(\xi)$ is
  antisymmetric in indices, so
  $v_{i_1\ldots i_d}(\xi)\neq 0$ if and only if all
  indices are distinct (see Section I.1 of \cite{KN63} for
  differential forms). In the local
  coordinate system, the components of the condition
  $\nabla^{(\alpha)}v=0$ gives the following system of
  partial differential equations (see Section III.7 of \cite{KN63}):
  \begin{equation}\label{al_cond}
    \partial_i v_{1\ldots d}-\sum_{j,k} {\Gamma^{(\alpha)}}^k_{ij}
    v_{1\ldots j-1kj+1\ldots d}
    =\partial_i v_{1\ldots d}-\sum_j{\Gamma^{(\alpha)}}^{j}_{ji}
    v_{1\ldots d}=0\quad
    \text{for~all} \quad i\in\{1,\ldots,d\},
  \end{equation}
  where the first equality follows by the antisymmetric
  property of $v_{i_1\cdots i_d}$. For simplicity,
  let $v_{1\ldots d}$ be denoted by $v$. If $\alpha=0$,
  the differential from with $v=\sqrt{g}$ is called
  the volume element and satisfies the system \eqref{al_cond}
  identically, because
  \[
    \sum_j{\Gamma^{(0)}}^j_{ij}=\sum_{j,k}g^{jk}\Gamma^{(0)}_{ij,k}
    =\frac{1}{2}\sum_{jk}g^{jk}\partial_ig_{jk}
    =\frac{1}{2} \partial_i\log g=\partial_i\log\sqrt{g},
  \]
  where the second equality follows by the expression (8).
  Therefore 0-connection is locally equiaffine and
  assertion (i) holds. For assertion (ii), if $M$ is
  $\alpha_0$-flat for a non-zero $\alpha_0$, we can
  take a local coordinate system satisfying
  ${\Gamma^{(\alpha_0)}}^j_{ij}=0$ around each point of
  $M$.
  By using (6), we have
  \[
  \sum_j{\Gamma^{(\alpha)}}^j_{ij}
  =\sum_j{\Gamma^{(\alpha_0)}}^j_{ij}+
  \frac{\alpha_0-\alpha}{2}\sum_{j,k}g^{jk}S_{ijk}
  \]
  for all $\alpha\in\mathbb{R}$. On the other hand,
  the second equality of (7) becomes
  \begin{equation}\label{g'}
  \partial_i g_{jk}=\alpha_0 S_{ijk}.
  \end{equation}
  Therefore, we have
  \[
    \sum_j{\Gamma^{(\alpha)}}^j_{ji}
    =\frac{\alpha_0-\alpha}{2\alpha_0}
    \sum_{j,k}g^{jk}\partial_i g_{jk}
    =\frac{\alpha_0-\alpha}{2\alpha_0}\partial_i \log g
  \]
  Then, the system \eqref{al_cond} becomes
  \[
    \partial_iv=\frac{\alpha_0-\alpha}{2\alpha_0}
    \partial_i\log g \quad \text{for~all} \quad i\in\{1,\ldots,d\}.
  \]
  This system is integrable, because it satisfies
  the integrability conditions (23),
  i.e., $\partial_i\partial_jh=\partial_j\partial_iv$ for all
  $i\neq j$. We immediately obtain
  $\log v=(1/2-\alpha/(2\alpha_0))\log g+\text{const.}$
  For assertion (iii), since the integrability condition yields
  $\partial_j({\Gamma^{(\alpha)}}^i_{ik}v)=\partial_k({\Gamma^{(\alpha)}}^i_{ij}v)$, $\forall \alpha,j\neq k$, we have
  $\partial_j{\Gamma^{(\alpha)}}^i_{ik}=\partial_k{\Gamma^{(\alpha)}}^i_{ij}$, $\forall \alpha,j\neq k$.
  On the other hand, by the definition of $\alpha$-connections
  (5), we have
  \[
    \sum_j{\Gamma^{(\alpha)}}^j_{ji}=\sum_j{\Gamma^{(0)}}^j_{ji}
    -\frac{\alpha}{2}S_i
    =\frac{1}{2}\partial_i\log g-\frac{\alpha}{2}S_i.
  \]
  Therefore, we have
  \begin{equation}\label{lemm:TA:1}
    \partial_iS_j=\partial_jS_i \quad
    \text {for all} \quad i\neq j.
  \end{equation}
  But we can construct a counter example, namely,
  there exists a manifold satisfying \eqref{lemm:TA:1}
  but is not $\alpha$-flat for any $\alpha$
  (see Section~4.3).
  \end{proof}
  
\begin{remark}
  Equiaffiness is a fandermental concept in affine differential
  geometry. Proposition I.3.1 of \cite{NS94} says that
  a torsion-free affine connection on $M$ is locally equiaffine
  if and only if the Ricchi tensor is symmetric. Since
  the 0-connection has the latter property, assertion (i)
  follows immediately.
  Takeuchi and Amari \cite{TA05} obtained assertion (ii) by using
  Corollary 3.12
  of \cite{Lau87}, but the proof above seems more straightforward.
\end{remark}

  Interestingly, $\alpha$-parallel volume elements have
  already appeared in various contexts in the statistical
  literature.
  
\begin{remark}\label{rema:TA}
  Section 6 of \cite{Har64} defined
  the asymptotically locally invariant prior by $v$ which
  solves the partial differential equation \eqref{al_cond}
  for $\alpha=1$, and satisfies $v(f(\xi))f'(\xi)\propto v(\xi)$
  asymptotically and locally for any estimand function $f$,
  with assuming the existence. He proposed a one-parameter
  family of invariant priors, which is a solution of
  \eqref{al_cond}. An example of $\alpha$-parallel
  volume elements with respect to an $\alpha_0$-flat
  manifold is $v(\alpha_0;\alpha_0)=1$ for any non-zero
  $\alpha_0$, where the case of the exponential family
  ($\alpha_0=1$) was discussed in \cite{Har98}. Another
  example is $v(\alpha;\pm 1)$, which appeared in
  \cite{TA05}.
  Kosmidis and Firth \cite{KF09} obtained
  $v(\alpha_0/2-1/2;\alpha_0)$ for any non-zero
  $\alpha_0$ as the penalty function for Firth's bias
  reduction method for generalized linear models for
  generic link function (Example 4).
\end{remark}

\begin{proof}[Proof of Corollary 3.2]
  Using the relation between $(-1)$ and $0$-Laplacian (11)
  with the expression of the $0$-Laplacian (10), 
  the partial differential equation (14) becomes
  \[
  \langle{\rm grad}\tilde{l},{\rm grad}f\rangle=
  -\frac{1}{2}\Delta^{(-1)}f=
  -\frac{1}{2}\frac{1}{\sqrt{g}}\sum_i\partial_i(\sqrt{g}({\rm grad}f)^i)
  +\frac{1}{4}\sum_iS_i({\rm grad}f)^i.
  \]
  When $d=1$, this reduces to the ordinary differential equation:
  \[
  g^{-1}\tilde{l}'f'=-\frac{1}{2}\frac{1}{\sqrt{g}}
  (g^{-1/2}f')'+\frac{1}{4}g^{-1}S_1f'=
  -\frac{f''}{2g}+\frac{f'}{4g^2}+\frac{f'}{4g}S_1,
  \]
  where we used $g^{11}=g^{-1}$ and $g_{11}=g$, and we have
  \begin{equation}\label{1dim}
  \tilde{l}'=-\frac{1}{2}\frac{f''}{f'}+\frac{1}{4}\frac{g'}{g}
  +\frac{1}{4}S_1
  \end{equation}
  and the integration yields (24). For the second assertion,
  if $\xi$ is an $\alpha$-affine coordinate system,
  \eqref{g'} gives $(\log g)'=\alpha S_1$, and \eqref{1dim}
  becomes
  \[
  \tilde{l}'=-\frac{1}{2}\frac{f''}{f'}+\frac{1+\alpha}{4\alpha}
  \frac{g'}{g}.
  \]
  Noting the definition of the $(\alpha-1)/2$-parallel volume element
  with respect to an $\alpha$-flat manifold (25), the integration
  yields (26).
\end{proof}

\begin{proof}[Proof of Corollary 3.4]
  The integrability condition (23) for the system of partial differential
  equations (28) is
\[
  \frac{1+\alpha}{2}(\partial_iS_j-\partial_jS_i)=
  \partial_i\left\{\frac{\Delta^{(\alpha)}f}{({\rm grad}f)^j}\right\}
  -\partial_j\left\{\frac{\Delta^{(\alpha)}f}{({\rm grad}f)^i}\right\}
  \quad \text{for~all} \quad i\neq j.
\]
If $\alpha\neq 0$, $S_i=\partial_i(\log g)/\alpha$ follows by
\eqref{g'}. The condition reduces to a simpler form:
\[
  \partial_i\left\{\frac{\Delta^{(\alpha)}f}{({\rm grad}f)^j}\right\}
  =\partial_j\left\{\frac{\Delta^{(\alpha)}f}{({\rm grad}f)^i}\right\}
  \quad \text{for~all} \quad i\neq j.
\]
\color{black}
An obvious class of estimand functions satisfying this integrability
condition is the $\alpha$-harmonic functions defined in Section~2.1,
that is, $\Delta^{(\alpha)}f=0$. The system of partial differential
equations (28) becomes
\[
  \partial_i\tilde{l}=\frac{1+\alpha}{4\alpha}\partial_i\log g \quad
  \text{for~all} \quad i\in\{1,\ldots,d\}
\]
and can be immediately integrated. The solution coincides with
the $(\alpha-1)/2$-parallel volume element with respect to
an $\alpha$-flat manifold. 
\end{proof}

\section{Estimating equations for the squared geodesic
  distance in the hyperbolic space} \label{appe:est_hyp}

The system of estimating equations for
$\hat{\xi}=(\hat{\mu},\hat{\sigma})$ which maximizes the penalized
likelihood \eqref{pen_hyp} is
\begin{align*}
  \mu&=\bar{x}+\frac{\sigma^2}{2n}
  \left(\frac{2\nu}{\nu^2-1}\frac{\mu}{\sigma}
  -\frac{1}{4f}\frac{\partial f}{\partial\mu}
  -\frac{1}{4g}\frac{\partial g}{\partial\mu}\right),\\
  \sigma^2&=2(s^2+(\bar{x}-\mu)^2)
  +\frac{\sigma^3}{n}
  \left\{\left(\frac{c}{4R^2}+2\right)\frac{1}{\sigma}
  +\frac{2(\sigma-\nu)\nu}{(\nu^2-1)\sigma}
  -\frac{1}{4f}\frac{\partial f}{\partial\sigma}
  -\frac{1}{4g}\frac{\partial g}{\partial\sigma}\right\},
\end{align*}
where
\begin{align*}
  f=&\mu^4\rho^2
    +\{\mu^2\rho^2+\sigma^2(\nu^2-1)^2\}(\sigma\nu-1)^2,\\
  g=&(\sigma\nu-1)^4\rho^2
     +\{(\sigma\nu-1)^2\rho^2+\sigma^2(\nu^2-1)^2\}\mu^2,\\
  \frac{\partial f}{\partial\mu}=&
  \frac{2\mu^3\rho}{\sigma\sqrt{\nu^2-1}}\{\mu^2+(\sigma\nu-1)^2\}
   +2\mu\{\mu^2(\sigma\nu+1)+(\sigma\nu-1)^2\}\rho^2\\
   &+2\mu\sigma(\sigma\nu-1)\{\sigma(3\nu^2-1)-2\nu\}(\nu^2-1),\\
  \frac{\partial g}{\partial\mu}=&
  \frac{2\mu(\sigma\nu-1)^2\rho}{\sigma\sqrt{\nu^2-1}}
  \{\mu^2+(\sigma\nu-1)^2\}
  +2\mu\sigma(\nu^2-1)\{2\mu^2\nu+\sigma(\nu^2-1)\}\\
  &+2\mu(\sigma\nu-1)\{\mu^2+(2\sigma\nu-1)(\sigma\nu-1)\}\rho^2,\\
  \frac{\partial f}{\partial\sigma}=&
  \frac{2\mu^2(\sigma-\nu)\rho}{\sigma\sqrt{\nu^2-1}}
    \{\mu^2+(\sigma\nu-1)^2\}+2\mu^2\sigma(\sigma\nu-1)\rho^2\\
    &+2\sigma(\sigma\nu-1)
    \{(\sigma\nu-1)(2\sigma\nu-\nu^2-1)+\sigma^2(\nu^2-1)\}(\nu^2-1),\\
  \frac{\partial g}{\partial\sigma}=&
  \frac{2(\sigma\nu-1)^2(\sigma-\nu)\rho}{\sigma\sqrt{\nu^2-1}}
  \{\mu^2+(\sigma\nu-1)^2\}+2\sigma(\sigma\nu-1)
  \{\mu^2+2(\sigma\nu-1)^2\}\rho^2\\
    &+2\mu^2(2\sigma\nu-\nu^2-1)\sigma(\nu^2-1).  
\end{align*} 

\end{appendix}

\section*{Acknowledgements}

The authors thank Professors Shiro Ikeda, Tao Zou, and
an anonymous referee for drawing their attention to
preceding works.
The first author was supported in part by JSPS KAKENHI
Grant 18K12758. The second author was supported in part
by JSPS KAKENHI Grants 18H00835 and 20K03742.

\end{document}